\documentclass[oneside,10pt,leqno]{amsart}
\usepackage{amssymb,amsmath,latexsym,amscd}
\usepackage{array,hhline}

\usepackage[T2A,T1]{fontenc}
\usepackage[utf8]{inputenc}
\usepackage[russian,english]{babel}

\usepackage{longtable}

\newcommand{\diag}{\operatorname{diag}}

\def\gg{\mathfrak{g}}
\def\gh{\mathfrak{h}}

\def\gl{\mathfrak{l}}

\def\gn{\mathfrak{n}}
\def\go{\mathfrak{o}}
\def\gp{\mathfrak{p}}

\def\gs{\mathfrak{s}}

\def\gu{\mathfrak{u}}
\def\gv{\mathfrak{v}}
\def\gw{\mathfrak{w}}

\def\gz{\mathfrak{z}}

\def\C{\mathbb{C}}

\def\F{\mathbb{F}}

\def\H{\mathbb{H}}

\def\O{\mathbb{O}}

\def\R{\mathbb{R}}

\def\Z{\mathbb{Z}}

\def\Im{{\rm Im}\,}
\def\Re{{\rm Re}\,}

\def\Ad{{\rm Ad}}

\def\Aut{{\rm Aut}}

\def\diag{{\rm diag}}

\renewcommand{\thesection}{\arabic{section}}

\renewcommand{\thetable}{{\large \thesection.\arabic{equation}}}
\newtheorem{theorem}[equation]{Theorem}

\newtheorem{lemma}[equation]{Lemma}

\newtheorem{corollary}[equation]{Corollary}

\newtheorem{proposition}[equation]{Proposition}

\newtheorem{definition}[equation]{Definition}

\newtheorem{example}[equation]{Example}

\newtheorem{remark}[equation]{Remark}

\def\sideremark#1{\ifvmode\leavevmode\fi\vadjust{\vbox to0pt{\vss
 \hbox to 0pt{\hskip\hsize\hskip1em
\vbox{\hsize2cm\tiny\raggedright\pretolerance10000 
 \noindent #1\hfill}\hss}\vbox to8pt{\vfil}\vss}}} 

\begin{document}

\title{Weakly Symmetric Pseudo--Riemannian Nilmanifolds}

\author{Joseph A. Wolf}\thanks{JAW was partially supported by the Simons
Foundation grant 346300 for IMPAN and the matching 2015-2019 Polish MNSW fund.}
\address{Department of Mathematics \\ University of California, Berkeley \\
	CA 94720--3840, U.S.A.} \email{jawolf@math.berkeley.edu}

\author{Zhiqi Chen}\thanks{ZC was partially supported by the National Natural 
Science Foundation of China (Grant Nos. 11571182 and 11931009)
and the Natural Science Foundation of Tianjin (Grant no. 19JCYBJC30600).}
\address{School of Mathematical Sciences and LPMC \\ Nankai University \\
	Tianjin 300071, P.R. China} \email{chenzhiqi@nankai.edu.cn}

\date{11 November 2018; referee suggestions incorporated 17 February 2020}



\begin{abstract}
In an earlier paper we developed
the classification of weakly symmetric pseudo--Riemannian manifolds
$G/H$ where $G$ is a semisimple Lie group and $H$ is a reductive subgroup.
We derived the classification from the cases where $G$ is compact.
As a consequence we obtained the classification of semisimple weakly
symmetric manifolds of Lorentz signature $(n-1,1)$ and trans--Lorentzian
signature $(n-2,2)$.  Here we work out the classification of
weakly symmetric pseudo--Riemannian nilmanifolds $G/H$ from the
classification for the case $G = N\rtimes H$ with $H$ compact and
$N$ nilpotent.  It turns out that there is a plethora of new examples 
that merit further study.
Starting with that Riemannian case, we see just when a 
given involutive automorphism of $H$ extends to an involutive automorphism 
of $G$, and we show that any two such extensions result in isometric 
pseudo--Riemannian nilmanifolds.  The results are tabulated in the last two
sections of the paper.
\end{abstract}

\maketitle

\section{Introduction}\label{sec1}
\setcounter{equation}{0}

There have been a number of important extensions of the theory of
Riemannian symmetric spaces.  Weakly symmetric spaces, introduced
by A. Selberg \cite{S1956}, play important roles in number theory, 
Riemannian geometry and harmonic analysis; see \cite{W2007}.  
Pseudo--Riemannian symmetric spaces also appear in
number theory, differential geometry, relativity, Lie group
representation theory and harmonic analysis.
We study the classification of weakly symmetric pseudo--Riemannian nilmanifolds.
This is essentially different from the Riemannian symmetric case; the
only simply connected Riemannian symmetric nilmanifolds are the Euclidean 
spaces.  The work of Vinberg, Yakimova and others shows that
there are many simply connected weakly symmetric 
Riemannian nilmanifolds, and here we will see that 
each of them leads to a large number of
simply connected weakly symmetric pseudo--Riemannian nilmanifolds.  

Recall that a Riemannian manifold $(M,ds^2)$ is weakly symmetric if,
for every $x \in M$ and $\xi \in M_x$\,, there is an isometry 
$s_{x,\xi}$ such that $s_{x,\xi}(x) = x$ and $ds_{x,\xi}(\xi) = -\xi$.  In
that case $(M,ds^2)$ is homogeneous.  Let $G$ be the identity component
of the isometry group, so $M$ has expression $G/H$.  Then the following
are equivalent, and each follows from weak symmetry.
(i) $M$ is $(G,H)$--commutative, i.e. the convolution algebra
$L^1(H\backslash G/H)$ is commutative, i.e. $(G,H)$ is a Gelfand pair,
(ii) the left regular representation of $G$ on $L^2(G/H)$ is
multiplicity free, and (iii) the algebra of
$G$--invariant differential operators on $M$ is commutative.
If $G$ is reductive then conversely commutativity implies weak symmetry, 
but that can fail for example when $G = N\rtimes H$ with $N$ nilpotent.  See
\cite{W2007} for a systematic treatment with proper references.

The situation is more complicated for pseudo--Riemannian weakly symmetric 
spaces.  If $(M,ds^2)$ is pseudo--Riemannian then the definition of
weakly symmetric can require only the existence of isometries $s_{x,\xi}$
as above, either for all $\xi \in M_x$ or only for $\xi \in M_x$ such
that $ds^2(\xi,\xi) \ne 0$. The definition, (i) above, of commutativity,
is problematic for noncompact $H$ because the convolution product involves 
integration over $H$.  The multiplicity free condition, (ii) above, becomes
$Hom_H(\pi_\infty,id) \leqq 1$ for every $\pi \in \widehat{G}$,
where $\pi_\infty$ is the space of $C^\infty$ vectors for $\pi$; see
the work
\cite{T1984}, \cite{vD1986}, \cite{vD2008}, \cite{vD2009} and \cite{vD2011}
of Thomas and van Dijk on generalized Gelfand pairs.
Our class of weakly symmetric pseudo--Riemannian manifolds consists of
the complexifications, and real forms of the complexifications, of
weakly symmetric Riemannian manifolds.  There is still quite a lot of 
work to be done to clarify the relations between these
various extensions of\, ``weakly symmetric Riemannian manifold'' and the
corresponding extensions of\, ``commutative space''.

The weakly symmetric spaces
we study have the form $G/H$ where $H$ is reductive in $G$ and $G$
is a semidirect product $N\rtimes H$ with $N$ nilpotent.  We find a large number
of interesting new examples of these spaces, in particular many new
homogeneous Lorentz and trans--Lorentz (e.g. conformally Lorentz) manifolds.

Our analysis of weakly symmetric pseudo--Riemannian nilmanifolds is based on
the work of (in chronological order) Wolf (\cite{W1963}, \cite{W1964}),
Carcano \cite{C1987}, Benson-Jenkins-Ratcliff (\cite{BJR1990},
\cite{BJR1992}, \cite{BJR1999}), Gordon \cite{G1996}, Vinberg (\cite{V2001},
\cite{V2003}) and Yakimova (\cite{Y2004}, \cite{Y2005}, \cite{Y2006}),
as described in \cite[Chapter 13]{W2007}.  Starting there we adapt the
``real form family'' method of Chen--Wolf \cite{CW2017} to the setting of
nilmanifolds.

We first consider weakly symmetric pseudo--Riemannian nilmanifolds
$(M,ds^2)$, $M = G/H$ with $G = N\rtimes H$, $N$ nilpotent and $H$ reductive
in $G$.  We show that every space of that sort belongs
to a family of such spaces associated to a weakly symmetric
Riemannian manifold $M_r = G_r/H_r$ ($r$ for Riemannian).  There $H_r$ is a
compact real form of the complex Lie group $H_\C$\,, $G_r = N_r\rtimes H_r$
is a real form of $G_\C$ and
where $N_r$ is a real form of $N_\C$\,, and $M_r$ is a weakly
symmetric Riemannian nilmanifold.  In fact, every weakly symmetric
Riemannian manifold is a commutative space, and we work a bit more generally,
assuming that $M$ and $M_r$ are commutative nilmanifolds.

\begin{definition}\label{family-def}
{\rm The  {\bf real form family} of $G_r/H_r$ consists of
$(G_r)_\C/(H_r)_\C$ and all $G'/H'$ with the same complexification
$(G_r)_\C/(H_r)_\C$.  The space $(G_r)_\C/(H_r)_\C$ is considered to be 
a real manifold.  We denote that real form family by $\{\{G_r/H_r\}\}$.
In this paper, a space $G/H \in \{\{G_r /H_r\}\}$ is called
{\bf weakly symmetric} just when $G_r /H_r$ is a weakly 
symmetric Riemannian manifold.} \hfill $\diamondsuit$
\end{definition}

\noindent
Our classifications will consist of examinations of the various possible
real form families.  
Proposition \ref{family-desc} reduces this to an
examination of involutions of the groups $G_r$\,.

\begin{proposition}\label{family-desc}
Let $G_r = N_r \rtimes H_r$ where $N_r$ is nilpotent, $H_r$ is compact,
and $M_r = G_r/H_r$ is a commutative, connected, simply connected
Riemannian nilmanifold.  Let $\sigma$ be an involutive
automorphism of $G_r$ that preserves $H_r$.  Then there is a unique
$G'/H'$ in the real form family $\{\{G_r/H_r\}\}$ such that $G'$ is
connected and simply connected, $H'$ is connected, and
$\gg' = \gg_r^+ +\sqrt{-1}\gg_r^-$ with $\gh' = \gh_r^+ +\sqrt{-1}\gh_r^-$
in terms of the $(\pm 1)$--eigenspaces of $\sigma$ on $\gg_r$ and $\gh_r$\,.
Up to covering, every space $G'/H' \in \{\{G_r/H_r\}\}$
either is obtained in this way or is the real manifold underlying the
complex structure of $(G_r)_\C/(H_r)_\C$\,.
\end{proposition}

\begin{proof}
First let $\sigma$ be an involutive automorphism of $G_r$ that preserves $H_r$.
In terms of the $(\pm 1)$--eigenspaces of $\sigma$ on $\gg_r$\,,
$\gg' = \gg_r^+ +\sqrt{-1}\gg_r^-$ is a well defined Lie algebra with nilradical
$\gn' = \gn_r^+ +\sqrt{-1}\gn_r^-$ and Levi component
$\gh' = \gh_r^+ +\sqrt{-1}\gh_r^-$\,.  Let $G'$ be the connected simply
connected
group with Lie algebra $\gg'$, and $H'$ and $N'$ the analytic subgroups
for $\gh'$ and $\gn'$\,.  Then $H'$ is reductive in $G'$, $N'$ is the
nilradical of $G'$ and is simply transitive on $G'/H'$, and
$G'/H' \in \{\{G_r/H_r\}\}$.

Conversely let $G'/H' \in \{\{G_r/H_r\}\}$; we want to construct the
involution $\sigma$ as above.  Without loss of generality we may assume
that $G'$ is connected and simply connected, that $H'$ is the
analytic subgroup for $\gh'$\,, and that $\gg'$ and $\gh'$ are stable under the
complex conjugation $\nu$ of $(\gg_r)_\C$ over $\gg_r$\,.  Then
$\gg' = (\gg')^+ + (\gg')^-$ and $\gh' = (\gh')^+ + (\gh')^-$,
eigenspaces under $\nu$.  Now $\gg_r = (\gg')^+ +\sqrt{-1} (\gg')^-$ and
$\gh_r = (\gh')^+ +\sqrt{-1} (\gh')^-$, and $\nu_{\gg_r}$ is the
desired involution.
\end{proof}

In Section \ref{sec2} we work out the relations between involutive
automorphisms of $H_r$ (which of course are known) and involutive
automorphisms of $G_r$\,.  This is a matter of understanding how to
extend an involutive automorphism of $H_r$ to an automorphism of $G_r$,
finding the condition for that extension to be involutive, and seeing that 
involutive extensions (when they exist) are essentially unique.

We need some technical preparation on linear groups and bilinear
forms in order to carry out our classifications.  That is carried 
out in Section \ref{sec3}.  Some of it is delicate.

In Section  \ref{sec4} we examine the case where $N_r$ is a 
Heisenburg group.  There are two distinct reasons for examining 
this Heisenberg case separately.  First, it indicates our general method 
and illustrates the need for the maximality condition when we look at 
a larger class of groups $N_r$\,.  But more important, the study of
harmonic analysis on $H_r\rtimes U(k,\ell)/U(k,\ell)$ is a developing
topic, and it is important to have a number of closely related examples.

Section \ref{sec5} contains our first main results.   We interpret
Vinberg's classification of maximal irreducible commutative
nilmanifolds, Table \ref{vin-table}, as the classification of all real
form families for all maximal irreducible commutative nilmanifolds.
Then we extend Vinberg's classification to a complete
analysis (including signatures of invariant pseudo--Riemannian metrics)
of the real form families of maximal irreducible commutative
nilmanifolds.  Many of these cases are delicate and rely on both 
the results of Section \ref{sec2} and specific technical
information worked out in Section \ref{sec3}.  Table \ref{max-irred} 
is the classification.  Then
we extract some non-Riemannian cases of special
interest from Table \ref{max-irred}.  Those are the cases of Lorentz signature
$(n-1,1)$, important in physical applications, and trans--Lorentz signature 
$(n-2,2)$, the parabolic geometry extension of conformal geometry.

Finally in Section \ref{sec6} we extend the results of Sections \ref{sec5}
to a larger collection of real form families, those where $G_r/H_r$ is 
indecomposable and satisfies certain technical conditions.  Those results
rely on the methods developed for the maximal irreducible case.  They 
are collected in Table \ref{indecomp}.  As corollaries we extract
the cases of Lorentz signature $(n-1,1)$ and trans--Lorentz signature
$(n-2,2)$.

\section{Reduction of the Real Form Question}\label{sec2}
\setcounter{equation}{0}

Proposition \ref{family-desc} reduces the classification of spaces in
a real form family $\{\{G_r/H_r\}\}$ to a classification of involutive
automorphisms of $G_r$ that preserve $H_r$\,.  If this section we
reduce it further to a classification of involutive automorphisms of
$H_r$\,.  For that we first review some facts about nilpotent groups that
occur in commutative Riemannian nilmanifolds.

Let $M_r = G_r/H_r$ be a commutative nilmanifold, $G_r = N_r \rtimes H_r$
with $N_r$ connected simply connected and nilpotent, and with $H_r$
compact and connected.  Then $N_r$ is the nilradical  of
$G_r$ and it is the only nilpotent subgroup of $G_r$ that is transitive
on $M_r$ \cite[Theorem 4.2]{W1963}\,.  Further, $N_r$ is commutative or
$2$--step nilpotent (\cite[Theorem 2.4]{BJR1990}, \cite[Theorem 2.2]{G1996}).
Thus we can decompose $\gn_r = \gz_r + \gv_r$ where $\gz_r$ is the center
and $\gv_r$ is an $\Ad(H_r)$--invariant complement.  Following
\cite[Chapter 13]{W2007} we say that
\begin{equation}\label{defs1}
\begin{aligned}
& G_r/H_r \text{ is {\em irreducible} if } [\gn_r,\gn_r] = \gz_r \text{ and }
        H_r \text{ is irreducible on } \gn_r/[\gn_r,\gn_r] \\
& (G_r/Z)/(H_r\cap Z) \text{ is a {\em central reduction} of } G_r/H_r
        \text{ if } Z \text{ is a closed central subgroup of } G_r \\
& G_r/H_r \text{ is {\em maximal } if it is not a nontrivial central
        reduction.}
\end{aligned}
\end{equation}

Now we employ a decomposition from \cite{V2001} and \cite{V2003}.  Split
$\gz_r = \gz_{r,0} \oplus [\gn_r,\gn_r]$ with $\Ad(H_r)\gz_{r,0} = \gz_{r,0}$\,.
Also decompose $\gn_r = \gz_{r,0} \oplus ([\gn_r,\gn_r] + \gv_r)$ with
$\gv_r = \gv_{r,1} + \dots + \gv_{r,m}$ vector space direct sum
where $\Ad(H_r)$ preserves and acts irreducibly on each $\gv_{r,i}$\,.
The representations of $\Ad(H_r)$ on the $\gv_{r,i}$ are mutually
inequivalent.
Consider the subalgebras $\gn_{r,i} = [\gv_{r,i},\gv_{r,i}] + \gv_{r,i}$ of
$\gn_r$ generated by $\gv_{r,i}$\,.  Then $[\gv_{r,i},\gv_{r,j}] = 0$
for $i \ne j$ and $(\xi_0,\xi_1,\dots,\xi_m) \mapsto (\xi_0+\xi_1+\dots+\xi_m)$
defines an $\Ad(H_r)$--equivariant homomorphism of
$\gz_r \oplus \gn_{r,1} \oplus \dots \oplus \gn_{r,m}$ onto $\gn_r$
with central kernel.

Let $N_{r,i}$ be the analytic subgroup of $N_r$ for $\gn_{r,i}$\,.
Let $H_{r,i}= H_r/J_{r,i}$ where $J_{r,i}$ is the
kernel of the adjoint action of $H_r$ on $\gn_{r,i}$\,.  Similarly let
$J_{r,0}$ be the kernel of the adjoint action of $H_r$ on $\gz_{r,0}$ and
$H_{r,0}= H_r/J_{r,0}$\,.  Let $G_{r,i} = N_{r,i}\rtimes H_{r,i}$ for
$i \geqq 0$.  As in  \cite[Section 13.4C]{W2007}, we summarize.

\begin{proposition}\label{factors}
The representation of $H_r$ on $\gv_{r,i}$ is irreducible.
If $i \ne j$ then $[\gn_{r,i},\gn_{r,j}] = 0$ and the representations
of $H_r$ on $\gv_{r,i}$ and $\gv_{r,j}$ are mutually
inequivalent.  $M_{r,0} = G_{r,0}/H_{r,0}$ is an Euclidean space, the other
$M_{r,i} = G_{r,i}/H_{r,i}$ are irreducible commutative Riemannian nilmanifolds.
\end{proposition}

The $\Ad(H_r)$--equivariant homomorphism of
$\gz_{r,0} \oplus \gn_{r,1} \oplus \dots \oplus \gn_{r,m}$ onto $\gn_r$ defines
a Riemannian fibration
\begin{equation}\label{ffactors}
\gamma : \widetilde{M_r} =  M_{r,0} \times M_{r,1} \times \dots \times M_{r,m}
	\to M_r
\end{equation}
with flat totally geodesic fibers defined by intersections of the
$\exp([\gv_{r,i},\gv_{r,i}])$.


\begin{theorem}\label{spec-rad}
Let $M_r = G_r/H_r$ be a commutative nilmanifold, $G_r = N_r \rtimes H_r$
with $N_r$ connected simply connected and nilpotent, and with $H_r$
compact and connected.
Let $G'/H'$, $G''/H'' \in \{\{G_r/H_r\}\}$.
If $H' \cong H''$, then $G' \cong G''$\,.
\end{theorem}
\begin{proof}
Let $\sigma', \sigma''$ be the involutive automorphisms of $G_r$ that
define $G'$ and $G''$.  As $H' \cong H''$, their restrictions to $\gh_r$
belong to the same component of the automorphism group.  Define
$L = \Ad(H_r) \cup \sigma'\Ad(H_r)$.  It is a compact group of linear
transformations of $\gg_r$ that has one or two components, and
$\sigma'' \in \sigma'\Ad(H_r)$.  Let $T$ be a maximal torus of the
centralizer of $\sigma'$ in $\Ad(H_r)$.  A theorem of de Siebenthal
\cite{deS1956} on compact disconnected Lie groups 
says that every element of $\sigma'\Ad(H_r)$ is $\Ad(H_r)$--conjugate to an
element of $\sigma' T$.  Thus we may replace $\sigma''$ by an
$\Ad(H_r)$--conjugate and assume $\sigma'' \in \sigma' T$.  That
done, $\sigma'$ and $\sigma''$ commute.  Thus we may assume that
$\nu := (\sigma')^{-1}\cdot\sigma''$ satisfies $\nu^2 = 1$.

We first consider the case where $G_r/H_r$ is irreducible.
In other words, following (\ref{defs1}), $[\gn_r,\gn_r] = \gz_r$
and $\Ad(H_r)|_{\gv_r}$ is irreducible.  Thus the commuting algebra
of $\Ad(H_r)|_{\gv_r}$ is $\R$, $\C$ or $\H$, so the only elements
of square $1$ in that commuting algebra are $\pm 1$.  As $\nu^2 = 1$, now
$\nu|_{\gv_r} = \pm 1$.  If
$\nu = 1$ on $\gv_r$ then $\nu = 1$ on $\gn_r$, in other words
$\sigma' = \sigma''$ on $\gn_r$\,.  Then $\gn' = \gn''$.  As
$\sigma'$ and $\sigma''$ commute, and as we have an isomorphism
$f: H' \cong H''$, we extend $f$ to an isomorphism $G' \to G''$ by the
identity on $N' = N''$.

The other possibility is that $\nu = -1$ on $\gv_r$. As linear
transformations of $\gv_r$, $\sigma' = c'$ and $\sigma'' = c''$
where $c'^2 = 1 = c''^2$ and $c'c'' = c''c'$.  Again using
irreducibility, $c' = \pm 1$ and $c'' = \pm 1$.  But $c'c'' =
\nu = -1$.  So we may suppose $c' = 1$ and $c'' = -1$.  Then
$\gn' = \gz_r + \gv_r$ and $\gn'' = \gz_r + \sqrt{-1}\gv_r$.
Now define $\varphi: \gn' \to \gn''$ by $\varphi(z,v) = (-z,\sqrt{-1}v)$.
Compute $[\varphi(z_1,v_1),\varphi(z_2,v_2)]$ =
$[(-z_1,\sqrt{-1}v_1),(-z_2,\sqrt{-1}v_2)]$ = $(-[v_1,v_2],0)$
= $\varphi([v_1,v_2],0)$ = $\varphi[(z_1,v_1),(z_2,v_2)]$.
Thus $\varphi: \gn' \to \gn''$ is an isomorphism.  It commutes with
$(\Ad(H_r)_\C)$, so it combines with $f: H' \cong H''$ to define an
isomorphism $G' \to G''$.

That completes the proof of Theorem \ref{spec-rad} in the case where
$G_r/H_r$ is irreducible.  We now reduce the general case to the
irreducible case, using Proposition \ref{factors}, i.e., material from
\cite[Section 13.4C]{W2007}.
As the representations
$\alpha_i$ of $H_r$ on the $\gv_{r,i}$ are inequivalent, $\nu|_{\gv_r}$
permutes the $\gv_{r,i}$\,.  If $\nu(\gv_{r,i}) = \gv_{r,j}$ with $i \ne j$
then $\nu$ defines an equivalence of $\alpha_i$ and $\alpha_j$\,,
contradicting inequivalence.  Thus $\nu(\gv_{r,i}) = \gv_{r,i}$ for every $i$.
As $\nu^2 = 1$ now $\nu|_{\gv_{r,i}} = \varepsilon_i = \pm 1$.  As in
the irreducible case $f:H' \cong H''$ together with the $\nu|_{\gv_{r,i}}$
defines an isomorphism $G' \cong G''$.
\end{proof}

\begin{theorem}\label{extends}
Let $M_r = G_r/H_r$ be a commutative, connected, simply connected
Riemannian nilmanifold, say with $G_r = N_r \rtimes H_r$ where $H_r$ is 
compact and connected and $N_r$ is nilpotent.  Let $\theta$ and $H$
denote an involutive automorphism of $H_r$ and the corresponding real
form of $(H_r)_\C$\,.  Consider the fibration
$\gamma: \widetilde{G_r}/\widetilde{H_r} = 
	\widetilde{M_r} \to M_r = G_r/H_r$
of {\rm (\ref{ffactors})}.  Then $\theta$ lifts to an involutive automorphism
$\widetilde{\theta}$ of $\widetilde{H_r}$\,, and $\widetilde{\theta}$ extends
to an automorphism $\sigma$ of $G_r$ such that $d\sigma(\gv_r) = \gv_r$\,.  

If $\sigma^2 = 1$ then the corresponding $(G,H) \in \{\{G_r/H_r\}\}$ is 
a homogeneous pseudo--Riemannian manifold.  If we cannot choose $\sigma$
so that $\sigma^2 = 1$ then $\theta$ and $H$ do not correspond to an
element of $\{\{G_r/H_r\}\}$.
\end{theorem}

\begin{proof} 
In the notation leading to Proposition \ref{factors}, $\gv_r =
\gv_{r,1} + \dots + \gv_{r,m}$ where $\Ad(H_r)$ acts on
$\gv_{r,i}$ by an irreducible representation $\alpha_i$\,, and the
$\alpha_i$ are mutually inequivalent.  As $\theta(H_r) =H_r$ the
corresponding representations $\alpha'_i = \alpha_i\cdot\theta$ just
form a permutation of the $\alpha_i$\,, up to equivalence.
If $\theta$ is inner then $\alpha'_i = \alpha_i$\,.

If $i \ne j$ with $\alpha_i$ equivalent to $\alpha'_j$\,,  let
$\tau$ denote the intertwiner.  So $\alpha_i'(h)\tau  =
\tau\alpha_j(h)$ and the intertwiner $\tau$ interchanges
$\gv_{r,i}$ and $\gv_{r,j}$\,.  On the other hand if $\alpha_i$ is
equivalent to $\alpha'_i$\,, the intertwiner $\tau$ satisfies
$\alpha_i'(h)\tau  = \tau\alpha_i(h)$ and $\tau\gv_{r,i}
= \gv_{r,i}$.  Thus $\alpha(\theta(h))\tau
= \tau\alpha(h)$ for $h \in H_r$\,.

Define $\widetilde{H_r} = H_r \cup tH_r$
where $tht^{-1} = \theta(h)$ and $t^2$ belongs to the center of $H_r$\,.  
Define $\sigma(h) = \alpha(h)$ and
$\sigma(th) = \tau\alpha(h)$ for $h \in H_r$ (in particular $\sigma(t) = \tau$).
We check that $\sigma$ is a representation of $\widetilde{H_r}$ on $\gv$:

(i) $\sigma(th_1)\sigma(th_2) = \tau\alpha(h_1) \tau\alpha(h_2)
= \alpha(\theta h_1)\alpha(h_2) = \alpha(th_1th_2) = \sigma(th_1th_2)$,

(ii) $\sigma(th_1)\sigma(h_2) = \tau\alpha(h_1)\alpha(h_2) =
\tau\alpha(h_1h_2) = \sigma(th_1h_2)$, and

(iii) $\sigma(h_1)\sigma(th_2) = \alpha(h_1)\tau\alpha(h_2) =
\tau\alpha(\theta h_1)\alpha(h_2) = \tau\alpha(\theta(h_1)h_2) =
\sigma(t\theta(h_1)h_2) = \sigma(h_1 th_2)$.

Now we check that $\sigma(\widetilde{H_r})$ consists of automorphisms of
$\gn_r$.  Set $\sigma(t)$ equal to the identity on $\gz_{r,0}$.
Since $[\gv_{r,i},\gv_{r,j}] = 0$ for $i \ne j$ we extend
$\sigma(t)$ to the subalgebras $\gz_{r,i} := [\gv_{r,i},\gv_{r,i}]$
by $\Lambda^2(\alpha_i)$.  In order that this be well defined on
$[\gn_r,\gn_r]$ it suffices to know that $[\gn_r,\gn_r] = [\gv_r,\gv_r]$
is the direct sum of the $\gz_{r,i} = [\gv_{r,i},\gv_{r,i}]$.  
That is clear if there is only
one $[\gv_{r,i},\gv_{r,i}]$, in other words if $\Ad(H_r)$ is irreducible
on $\gv_r$\,.  In general $\theta$ permutes the irreducible
factors of the representation of $H_r$ on $\gv_r$, so it permutes the
$\gv_{r,i}$\,. Thus $\theta$ lifts to $\widetilde{H_r}$\,,
and we apply the irreducible case result to the factors $M_{r,i}$\,.
Thus $\theta$ extends to the automorphism $\sigma$ of $G_r$\,.
\end{proof}

\begin{remark}\label{vh-zvh}{\rm
If $\theta$ extends to $\sigma \in \Aut(G_r)$ then evidently that extension
is well defined on $\gv_r\rtimes H_r$\,.  But the converse holds as well
(and this will be important for us): If $\theta$ extends to $\alpha \in
\Aut(\gv_r\rtimes H_r)$ then $\theta$ extends to an element of $\Aut(G_r)$.
For $\alpha$ extends to $(\gz_r + \gv_r)\rtimes H_r$ since $\gz_r$ is an
$\Ad(H_r)$--invariant summand of $\Lambda^2_\R(\gv_r)$, and that extension
of $\alpha$ exponentiates to some $\sigma \in \Aut(G_r)$.}\hfill $\diamondsuit$
\end{remark}

\begin{corollary}\label{fam-ps-riem}
Let $M_r = G_r/H_r$ be a commutative, connected, simply connected
Riemannian nilmanifold, say with $G_r = N_r \rtimes H_r$ where $H_r$ is 
compact and connected and $N_r$ is nilpotent.  Let $M = G/H$ belong to
the real form family $\{\{G_r/H_r\}\}$.  Then $G = N\rtimes H$ where
$\gn = \gz + \gv$, $\gz$ is the center, and each of $\gn$ and $\gz$
has a nondegenerate $Ad(H)$--invariant symmetric bilinear form.  In
particular $M = G/H$ is a pseudo--Riemannian homogeneous space.
\end{corollary}

\begin{proof}
By Proposition \ref{family-desc} and Theorem \ref{extends},
the pair $(G,H)$ corresponds to an involutive automorphism $\sigma$ of
$G_r$ whose complex extension and restriction to $G$ gives a Cartan
involution of $H$.  We may assume that $\sigma$ preserves the 
nilradical $\gn_r$ of $\gg_r$, the center $\gz_r$ of $\gn_r$, and
the orthocomplement $\gv_r$ of $\gz_r$ in $\gn_r$.  Thus, in the
Cartan duality construction described in Proposition \ref{family-desc}, 
the positive
definite $\Ad(H_r)$--invariant inner product on $\gn_r$ (corresponding
to the invariant Riemannian metric on $G_r/H_r$), gives us a
nondegenerate $\Ad(H)$--invariant symmetric bilinear form on $\gn$
for which $\gv = \gz^\perp$.  The corollary follows.
\end{proof}

\begin{example}\label{heis}{\rm
Consider the Heisenberg group case $\gn_r = \Im\C + \C^n$ and
$\gh_r = \gu(n)$, with $\theta(h) = \overline{h}$.  Then $\theta$
extends to an involutive automorphism $\sigma$ of $\gg_r$ by complex
conjugation on $\gn_r$\,.
Denote $\widetilde{H_r} = H_r \cup tH_r$ where $t^2 = 1$ and
$tht^{-1} = \theta(h)$ for $h \in H_r$\,.  Then $\sigma(th)n =
\sigma(tht^{-1}\cdot t)n = \sigma(\theta(h))\sigma(t)n =
\sigma(\overline{h})\overline{n} = \sigma(t)(\sigma(h)n)$.  Thus in fact
$\sigma$ defines a representation of $\widetilde{H_r}$ given, in this
Heisenberg group case, by $\sigma(t): n \mapsto \overline{n}$.
}\hfill$\diamondsuit$
\end{example}

\section{Irreducible Commutative Nilmanifolds: Preliminaries}\label{sec3}
\setcounter{equation}{0}

Recall the definition (\ref{defs1}) of maximal irreducible commutative
Riemannian nilmanifolds.  They were classified by Vinberg 
(\cite{V2001}, \cite{V2003}) (or see \cite[Section 13.4A]{W2007}), and
we are going to extend that classification to the pseudo--Riemannian
setting.  In order to do that we need some specific results on linear
groups and bilinear forms.  We work those out in this section, and
we extend the Vinberg classification in the next section. 
\medskip

\centerline{$\mathbf{U(1)}$ {\bf factors.}}

We first look at the action of $\theta$ when $H_r$ has a $U(1)$ factor 
and see just when $\theta$ extends to an involutive automorphism of
$\gg_r$, in other words just when we do have a corresponding 
$(G,H)$ in $\{\{G_r,H_r\}\}$.

\begin{lemma}\label{ext1}
Let $(G_r,H_r)$ be an irreducible commutative Riemannian nilmanifold
such that $H_r = U(1)\cdot H'_r$\,.  Suppose that $|U(1)\cap H_r'| \geqq 3$.  
Let $(G,H) \in \{\{G_r,H_r\}\}$ corresponding to an involutive 
automorphism $\theta$ of $H_r$\,.  If $\theta|_{H_r'}$ is inner then
$H$ has form $U(1)\cdot H'$.  If $\theta|_{H_r'}$ is outer then
$H$ has form $\R^+\cdot H'$.
\end{lemma}

\begin{lemma}\label{ext2}
Let $(G_r,H_r)$ be an irreducible commutative Riemannian nilmanifold
such that $H_r = U(1)\cdot H'_r$\,.  Suppose that $|U(1)\cap H_r'| \leqq 2$.
Let $\theta'$ be an involutive automorphism of $H_r'$ and $H'$ the
corresponding real form of $(H'_r)_\C$\,.  
Then $\{\{G_r,H_r\}\}$ contains both an irreducible commutative 
pseudo--Riemannian nilmanifold with $H = U(1)\cdot H'$ and
an irreducible commutative pseudo--Riemannian nilmanifold with 
$H = \R^+\cdot H'$.
\end{lemma}

For all $G_r/H_r$ in Vinberg's list (\ref{vin-table}), for which 
$H_r = (U(1)\cdot) H'_r$\,,
the representation of $H_r$ on $\gv_r$ is not absolutely irreducible.
In other words $(\gv_r)_\C$ is of the form $\gw_r \oplus \overline{\gw_r}$
in which $\gv_r$ consists of the $(w,\overline{w})$.  Thus $u \in U(1)$
acts by $(w,\overline{w}) \mapsto (uw, \overline{uw})$.  In consequence, 

\begin{lemma}\label{ext3}
If $(G,H) \in \{\{G_r,H_r\}\}$ with $H = \R^+\cdot H'$ then 
$\gv_r = \gv_r' \oplus \gv_r''$ direct sum of real vector spaces 
that are eigenspaces of $\R^+$, in other words by the condition that
$e^t \in \R^+$ acts on $\gv_r$ by $v'+v'' \mapsto e^t v'+ e^{-t}v''$.
In particular $\gv'$ and $\gv''$ are totally isotropic, and paired with
each other, for the $Ad(H)$--invariant inner product on $\gv_r$\,.
Consequently that invariant inner product has signature $(p,p)$ where
$p = \dim_\C \gv_r = \tfrac{1}{2}\dim_\R \gv_r$\,.
\end{lemma}
\medskip

\centerline{\bf Spin Representations}

Next, we recall signatures of some spin representations for the groups 
$Spin(k,\ell)$.  

\begin{lemma}\label{spin} {\rm (\cite[Chapter 13]{H1990})}
Real forms of $Spin(7;\C)$ satisfy
$$
Spin(6,1) \subset SO^*(8) \simeq SO(6,2), 
Spin(5,2) \subset SO^*(8) \simeq SO(6,2) 
\text{ and } Spin(4,3) \subset SO(4,4).
$$
Real forms of $Spin(9;\C)$ satisfy
$$
Spin(8,1) \subset SO(8,8), 
Spin(7,2) \subset SO^*(16),
Spin(6,3) \subset SO^*(16) 
\text{ and } Spin(5,4) \subset SO(8,8).
$$
Real forms of $Spin(10;\C)$ satisfy
$$
\begin{aligned}
&Spin(9,1) \subset SL(16;\R), 
Spin(8,2) \subset SU(8,8),
Spin(7,3) \subset SL(4;\H) \subset Sp(4,4), 
Spin(6,4) \subset SU(8,8),\\
& Spin(5,5) \subset SL(16;\R) \subset SO(16,16),
\text{ and } Spin^*(10) \subset SL(4;\H) \subset Sp(4,4). 
\end{aligned}
$$ 
\end{lemma}
\medskip

\centerline{$\mathbf{E_6}$}

Issues involving $E_6$ are more  delicate.
If $L$ is a connected reductive Lie group, let $\varphi_{L,b}$
denote the fundamental representation corresponding to the
$b^{th}$ simple root in Bourbaki order, $\varphi_{L,0}$
denote the trivial $1$--dimensional representation, and write $\tau$ for
the defining $1$--dimensional representation of $U(1)$.  Then
$\varphi_{E_6,6}|_{C_4} = \varphi_{C_4,2}$,
$\varphi_{E_6,6}|_{F_4} = \varphi_{F_4,4}\oplus \varphi_{F_4,0}$ and
$\varphi_{E_6,6}|_{A_5A_1} = (\varphi_{A_5,5}\otimes \varphi_{A_1,1})
	\oplus (\varphi_{A_5,2}\otimes \varphi_{A_1,0})$. 
These $C_4$ and $F_4$ restrictions are real.
As $\varphi_{E_6,6}(H)$ is noncompact, and there are only one or two 
summands under its maximal compact subgroup, we conclude
that $\varphi_{E_6,6}(E_{6,C_4}) \subset SO(27,27)$,
$\varphi_{E_6,6}(E_{6,F_4}) \subset SO(26,1)$ and
$\varphi_{E_6,6}(E_{6,A_5A_1}) \subset SU(15,12)$.     

$\varphi_{E_6,6}|_{D_5T_1}$ has three summands, 
$(\varphi_{D_5,0}\otimes \tau^{-1})
\oplus (\varphi_{D_5,4}\otimes \tau^{-1})
\oplus (\varphi_{D_5,1}\otimes \tau^{2})$, of respective degrees
$1$, $16$ and $10$, so the above argument must be supplemented.
For that, we look at $\varphi_{E_6,6}|_{L_r}$ where ${L_r} \cong SU(3)$
is a certain subgroup of $E_6$\,.

Write $\xi_b$ for the $b^{th}$ fundamental highest weight of $A_2$\,.
Thus $A_2$ has adjoint representation $\alpha := \varphi_{A_2, \xi_1+\xi_2}$
Denote $\beta:= \varphi_{A_2, 2\xi_1 + 2\xi_2}$, so the symmetric square
$S^2(\alpha) = \varphi_{A_2,0} \oplus \alpha \oplus \beta$.  Then
\cite[Theorem 16.1]{D1952} $E_6$ has a subgroup ${L_r}\cong SU(3)$ such that
$\varphi_{E_6,6}|_{L_r} = \beta$.  Further
\cite[Table 24]{D1952} ${L_r}$ is invariant under the outer automorphism of
$E_6$ that interchanges $\varphi_{E_6,6}$ with its dual $\varphi_{E_6,1}$\,.
The representation of $H_r$ on $\gv_r$ treats $\gv_r$ as the unique
$(\varphi_{E_6,6}\oplus \varphi_{E_6,1})$--invariant of $\gv_\C$\,, so it
is the invariant real form for
$\varphi_{E_6,6}({L_r})\oplus \varphi_{E_6,1}({L_r})$\,.
Thus the representation of $H = E_{6,D_5T_1}$ on $\gv$ treats $\gv$ as the
invariant real form of $\gv_\C$ for
$\varphi_{E_6,6}(L)\oplus \varphi_{E_6,1}(L)$
where $L = (L_r)_\C \cap H$.  $L$ must be one of the real forms $SU(1,2)$ or
$SL(3;\R)$ of $(L_r)_\C = SL(3;\C)$.  Now
$S^2(\alpha)$ has signature $(20,16)$ or $(21,15)$.  Subtracting
$\alpha$ from $S^2(\alpha)$ leaves signature $(16,12)$, and subtracting 
$\varphi_{A_2,0}$ (for the Killing form of $L$) leaves signature $(15,12)$ 
or $(16,11)$.  But this
must come from the summands of $\varphi_{E_6,6}|_{D_5T_1}$, which have
degrees $1$, $16$ and $10$.  Thus $\varphi_{E_6,6}(E_{6,D_5T_1})
\subset SU(16,11)$.
\medskip

\centerline{\bf Split Quaternion Algebra}

Another delicate matter concerns the split real quaternion algebra
$\H_{sp}$\,.  While $\H_{sp} \cong \R^{2\times 2}$, the conjugation of
$\H_{sp}$ over $\R$ is given by 
$\overline{\bigl ( \begin{smallmatrix} a & b \\ c & d \end{smallmatrix} \bigr )}
= \bigl ( \begin{smallmatrix} d & -b \\ -c & a \end{smallmatrix} \bigr )$.
Thus $\Im\H_{sp}^{n\times n}$ has real dimension $2n^2+n$ and is isomorphic
to the Lie algebra of $Sp(n;\R)$, and 
$\Re\H_{sp}^{n\times n}$ has real dimension $2n^2-n$.  In Case 9 of Table
\ref{max-irred} below, we can have $\gv = \H^n_{sp}$ with 
$\gz = \Re\H_{sp,0}^{n\times n} \oplus \Im\H_{sp}$\,, where 
$H = \{1,U(1),\R^+\}Sp(n;\R)$.  Then the bracket $\gv \times \gv \to \gz$
has two somewhat different pieces.  The obvious one is $\gv \times \gv \to
\Im\H_{sp}$\,, given by $[u,v] = \Im u\overline{v}$\,.  For the more subtle one
we note $\H^n \simeq \C^{2n}$ as a $\C^*\cdot Sp(n;\C)$--module, 
$\gz = [\gv,\gv] \subset \Lambda^2_\C(\C^{2n})$, and 
$\dim_\R \Re\H_{sp}^{n\times n} = 2n^2-n = \dim_\C(\Lambda^2_\C(\C^{2n}))$,
so $\Re\H_{sp}^{n\times n}$ is an $Sp(n;\R)$--invariant real form of 
$\Lambda^2_\C (\C^{2n})$.  Thus 
$\Re\H_{sp}^{n\times n} \simeq \Lambda^2_\R(\R^{2n})$
as an $Sp(n;\R)$--module.  If $\omega$ denotes the $Sp(n;\R)$--invariant
antisymmetric bilinear form on $\R^{2n}$ then $\langle u\wedge v,
u'\wedge v'\rangle = \omega(u,u')\omega(v,v')$ defines the 
$Sp(n;\R)$--invariant symmetric bilinear form on 
$\Lambda^2_\R(\R^{2n})$.  The assertions in Case 9 of Table
\ref{max-irred} follow.
\medskip

\centerline{$\mathbf{SL(n/2;\H)}$ {\bf and} {$\mathbf{GL(n/2;\H)}$}}

The real forms $SL(n/2;\H)$ of $SU(n)$ and $GL(n/2;\H)$ of $U(n)$ can
appear or not, in an interesting way.

\begin{lemma}\label{no-glH}
Let $(G,H) \in \{\{G_r,H_r\}\}$.
Suppose that $H_r = U(n)$ or $H_r = SU(n)$, and that $\gv_r = \C^n$. Then 
$H \ne GL(n/2;\H)$ and $H \ne SL(n/2;\H)$.
\end{lemma}

\begin{proof}
Let $H_r$ be $U(n)$ or $SU(n)$ and $\gv_r = \C^n$.
Suppose that $H$ is $GL(n/2;\H)$ or $SL(n/2;\H)$.  Then
the corresponding involutive automorphism $\theta$ of $H_r$ is given
by $\theta(g) = J\overline{g}J^{-1}$ where $J = \diag\{J',\dots,J'\}$
with $J' = (\begin{smallmatrix} 0 & 1 \\ -1 & 0\end{smallmatrix})$.  
Then one extension of $\theta$ to $G_r = N_r \rtimes H_r$ is given 
on $G_r/Z_r \simeq \gv_r\rtimes H_r$ by 
$\alpha(x,g) = (J\overline{x},J\overline{g}J^{-1})$.  This extension
is not involutive: $\sigma^2(x,1) = (-x,1)$.  However, since
$H$ is $GL(n/2;\H)$ or $SL(n/2;\H)$, $\theta$ has an involutive
extension $\beta$ to $G_r$\,.  Thus 
$\beta(x,g) = (B\overline{x},J\overline{g}J^{-1})$ for some $B \in U(n)$.
We compare $\alpha$ and $\beta$.

Calculate $\beta(x,g)\beta(x',g') = (B\overline{x},J\overline{g}J^{-1})
(B\overline{x'},J\overline{g'}J^{-1})  
= (B\overline{x} + J\overline{g}J^{-1}(B\overline{x'}), J\overline{gg'}J^{-1})$
and $\beta((x,g)(x',g')) = \beta(x+g(x'),gg') =
(B\overline{x} + B\overline{gx'},J\overline{gg'}J^{-1})$.  Since $\beta$
is an automorphism this says $J\overline{g}J^{-1}B = B\overline{g}$, in
other words $\overline{g}\cdot J^{-1}B = J^{-1}B \cdot \overline{g}$.
Thus $J^{-1}B$ is a central element of $U(n)$, in other words
$B = cJ$ with $c \in \C\,, |c| = 1$.  As $\beta$ is involutive we calculate
$(x,g) = \beta^2(x,g) = \beta(B\overline{x}, J\overline{g}J^{-1})
= (B\,\overline{B\overline{x}},J\,\overline{J\overline{g}J^{-1}}\,J^{-1})
= (B\overline{B}x,g)$.  Thus $I = B\overline{B} = (cJ)(\overline{cJ})
= |c|^2J^2 = -I$.  That contradicts 
Theorem \ref{extends}, and the Lemma follows.
\end{proof}

\begin{remark}\label{s2-l2-ok}{\rm
Let $(G,H) \in \{\{G_r,H_r\}\}$ with $H_r = U(n) \text{ or } SU(n)$.
Suppose that $H$ is $GL(n/2;\H)$ or $SL(n/2;\H)$ as defined by the 
involutive automorphism $\theta$ of $H_r$\,.  Then 
$\theta(g) = J\overline{g}J^{-1}$ as noted in the proof of Corollary 
\ref{no-glH}.  By contrast here, if $\gv_r$ is a subspace of 
$\bigotimes^2(\C^n)$ then $\theta$ does extend to an involutive
automorphism of $G_r$\,.  That extension is given on $\gv_r$ by
$\alpha(x\otimes y,g) = 
(J\overline{x}\otimes J\overline{y},J\overline{g}J^{-1})$\,, and
on $\gz_r$ as a subspace of $\Lambda^2_\R(\gv_r)$.  
The point here is that $\alpha^2(x\otimes y,g)
= (J\overline{J\overline{x}}\otimes J\overline{J\overline{y}},
J\overline{J\overline{g}J^{-1}}J^{-1} = (J^2x\otimes J^2y,g)
= ((-x)\otimes(-y),g) = (x\otimes y,g)$.}\hfill$\diamondsuit$
\end{remark}

\begin{remark}\label{some-prods-ok}{\rm
Following the idea of Remark \ref{s2-l2-ok}, suppose that $H_r$ is locally
isomorphic to a product, say $H_r = H_r'\cdot H_r''$ with 
$H_r'' = U(n) \text{ or } SU(n)$, and that $\gv_r = \gv_r'\otimes\gv_r''$
accordingly with $\gv_r'' = \C^n$.  Suppose $\theta = \theta'\otimes\theta''$
so that $H$ splits the same way.  Let $\sigma'$ denote the extension of
$\theta'$ to $\gv_r'$\,.  If $\sigma'^2 = 1$ then $H''$ cannot be 
$GL(n/2;\H)$ or $SL(n/2;\H)$.  For example this says that $H$ cannot
be $SL(m;\R)\times SL(n/2;\H)$. But if $\sigma'^2 = -1$ and 
$\gv_r = \C^m\otimes\C^n$, one must consider the possibility
that $H$ be $SL(m/2;\H)\cdot SL(n/2;\H)$.}\hfill$\diamondsuit$
\end{remark}
\medskip

\centerline{$\mathbf{U(1)\cdot H_r''}$}

\begin{remark}\label{some-prods-maybe}{\rm
A small variation the argument of Remark \ref{some-prods-ok} has a useful
application to some more of the the cases 
$H_r = H_r''\cdot H_r'''$ where $H_r''$ has form  $U(1)\cdot H_r'$\,.
If $\theta|_{H_r''}$ has form $g \mapsto JgJ^{-1}$ with $J^2 = -I$,
so that the extension $x \mapsto Jx$ to $\gv_r$ has square $-I$, we
can replace $J$ by $iJ$; then the extension $x \mapsto iJx$ is involutive.
If $\theta|_{H_r''}$ has form $g \mapsto J\overline{g}J^{-1}$ with $J$ real
and $J^2 = -I$, then this fails, for if $c \in U(1)$
the extension $x \mapsto cJ\overline{x}$ has square -I, thus is not involutive.
That could be balanced if extension of $\theta|_{H_r'''}$ also
has square $-I$.
}\hfill$\diamondsuit$
\end{remark}
\medskip

\centerline{$\mathbf{SO^*(2\ell)}$ {\bf and } $\mathbf{Sp(n;\R)}$}

The analog of Lemma \ref{no-glH} (or at least the analog of the proof) 
for the groups $SO^*(2\ell)$ is

\begin{lemma}\label{no-so*}
Let $(G,H) \in \{\{G_r,H_r\}\}$ with $H_r = SO(2\ell)$ and $\gv_r = \R^{2\ell}$.
Then $H \ne SO^*(2\ell)$.
\end{lemma}

\begin{proof}
Suppose $H = SO^*(2\ell)$.  Then it is the centralizer in $SO(2\ell)$ of
$J = (\begin{smallmatrix} 0 & I \\ -I & 0 \end{smallmatrix})$, and
$\theta(g) = JgJ^{-1}$.  Thus $\theta$ extends to an automorphism $\alpha$
of $G_r$ given on $G_r/Z_r \simeq \gv_r\rtimes H_r$ by 
$\alpha(x,g) = (Jx,JgJ^{-1})$.
Note that $\alpha^2$ is $-1$ on $\gv_r$\,.  Now let $\beta$ be an involutive
extension of $\theta$, so $\beta(x,g) = (Bx,JgJ^{-1})$ with $B^2 = I$.
Compute $\beta(x,g)\beta(x',g') = (Bx+JgJ^{-1}Bx',Jgg'J^{-1})$ and
$\beta( (x,g)(x',g') ) = \beta(x+gx', gg') = (Bx+Bgx',Jgg'J^{-1})$,
so $JgJ^{-1}Bx' = Bgx'$ and it follows that $J^{-1}B$ is central in
$SO(2\ell)$.  Thus either $J^{-1}B = I$ so $J = B$ contradicting
$J^2 = -I = -B^2$, or $J^{-1}B = -I$ so $J = -B$ contradicting
$J^2 = -I = -B^2$.  That contradicts $\beta^2 = 1$, and the Lemma follows.
\end{proof}

The arguments of Lemmas \ref{no-glH} and \ref{no-so*} go through without 
change for $Sp(n;\R)$:

\begin{lemma}\label{no-spnr}
Let $(G,H) \in \{\{G_r,H_r\}\}$ with $H_r = Sp(n)$ and $\gv_r = \C^{2n}$.
Then $H \ne Sp(n;\R)$.
\end{lemma}

\begin{proof} $Sp(n;\R)$ is defined by the involution 
$\theta: g \mapsto JgJ^{-1}$ of $Sp(n)$ with fixed point set $U(n)$.
As before, $\theta$ extends to an automorphism $\alpha$ on $G_r$
given on $G_r/Z_r \simeq \gv_r\rtimes H_r$ by 
$\alpha(x,g) = (Jx,JgJ^{-1})$, where
$J = (\begin{smallmatrix} 0 & I \\ -I & 0 \end{smallmatrix})$, we let
$\beta$ be an involutive extension of $\theta$, and note 
$\beta(x,g) = (Bx,JgJ^{-1})$ with $B^2 = I$.  We do a computation to see that 
$J^{-1}B$ is central in $Sp(n)$, so it is $\pm I$, and $J = \pm B$.  Either
choice of sign contradicts $B^2 = I$, and the Lemma follows.
\end{proof}

\begin{remark}\label{some-prods}{\rm
Consider cases where the semisimple part of $H_r$ is of the form
$H'_r\cdot H''_r$ with $H''_r = Sp(n)$.  Let $\theta = \theta' \times \theta''$
define the real form $H'' = Sp(n;\R)$ corresponding to $\theta''$.
Consider an extension $\beta$ of $\theta$ to $G_r$, given by
$\beta(x'\otimes x'',g) = (B'x'\otimes B''x'',\theta g)$ on $\gv_r\times H_r$\,.
The argument of Lemma \ref{no-spnr} shows $B''^2 = -I$, so $\beta$ cannot be
involutive unless either $B'^2 = -I$ as well, or $H_r$ has a $U(1)$ factor so 
that we can replace $B'$ by a scalar multiple with square $-I$.  For example, 
in Table \ref{max-irred},

In Case 9: $H \ne Sp(n;\R)$ and $H \ne \R^+\cdot Sp(n;\R)$. 

In Case 17: $H \ne Sp(1)\times Sp(n;\R)$ and $H \ne Sp(1;\R)\otimes Sp(r,s)$.

In Case 18:
$H \notin \{Sp(2)\times Sp(n;\R), Sp(2;\R)\times Sp(r,s), 
Sp(1,1)\times Sp(n;\R)\}$.  

In Case 21: $H\notin \{\{1,\R^+\}
(Sp(2;\R)\times SL(n;\R)), (Sp(2;\R)\times SU(r,s)\}$.  

In Case 22: $H\ne GL(2;\R)\times Sp(n;\R)$.

In Case 23: $H\ne GL(3;\R)\times Sp(n;\R)$.

\noindent Similar methods and restrictions apply to Table \ref{indecomp}.
} \hfill $\diamondsuit$
\end{remark}

\centerline{{\bf Signature of Products}}

\begin{lemma}\label{more-prods}
Suppose that $H_r$ is irreducible on $\gv_r$ and that 
$H_r = H'_r\cdot H''_r$ with $\theta = \theta' \times \theta''$.
Suppose further that there is an involutive extension of $\theta$
to $G_r$\,, resulting in $(G,H) \in \{\{(G_r,H_r\}\}$ with $H = H'\cdot H''$.
Further, $\gv$ has form $\gv'\otimes\gv''$ with action of $H$ of the form 
$\alpha' \otimes \alpha''$, with invariant $\R$--bilinear forms $b'$ and $b''$.
The $H$--invariant symmetric $\R$--bilinear form on $\gv$ is 
$b := b' \otimes b''$. 
If one of $b', b''$ is antisymmetric, so is the other, and $b$ has
signature $(t,t)$ where $2t = \dim_\R(\gv)$.  If one of $b', b''$ has
signature of the form $(u,u)$ then also $b$ has
signature $(t,t)$ where $2t = \dim_\R(\gv)$.
More generally, if $b'$ is symmetric with signature $(k,\ell)$
and $b''$ is symmetric with signature $(r,s)$ then $b$ has signature
$(kr + \ell s, ks + \ell r)$.
\end{lemma}

\begin{proof} Since $H_r$ is irreducible on $\gv_r$ its action there has
form $\alpha'_r\otimes\alpha''_r$\,.  Thus the action of $H$ on $\gv$ has
form $\alpha'\otimes\alpha''$, and the $H$--invariant symmetric $\R$--bilinear
form on $\gv$ is $b := b' \otimes b''$ as asserted.  

If $b'$ is
antisymmetric, then $b''$ is antisymmetric also, because $b$ is symmetric.
Then $\gv' = \gv'_1 \oplus \gv'_2$ where $b'(\gv'_i,\gv'_i) = 0$ and
$b'$ pairs $\gv'_1$ with $\gv'_2$\,, and $\gv'' = \gv''_1 \oplus \gv''_2$
similarly.  Choose bases $\{e'_{1,i}\}$ of $\gv'_1$\,, $\{e'_{2,j}\}$ of 
$\gv'_2$\,, $\{e''_{1,u}\}$ of $\gv''_1$ and $\{e''_{2,v}\}$ of $\gv''_2$
such that $b'(e'_{1,i}, e'_{2,j}) = \delta_{i,j}$ and
$b''(e''_{1,u}, e''_{2,v}) = \delta_{u,v}$\,.  Here of course
$\dim \gv'_1 = \dim \gv'_2$ and $\dim \gv''_1 = \dim \gv''_2$\,.  
Then $(\gv'_1\otimes \gv''_1)
\oplus (\gv'_2\otimes \gv''_2)$ is positive definite for $b$,
$(\gv'_1\otimes \gv''_2) \oplus (\gv'_2\otimes \gv''_1)$ is negative
definite for $b$, and the two are $b$--orthogonal and of equal dimension.
That proves the first assertion on signature.

Now suppose that $b'$ and $b''$ are symmetric, that $b'$ has signature
of the form $(u,u)$, and that $b''$ has signature of the form $(v,w)$.
Then $\gv' = \gv'_1 \oplus \gv'_2$ into orthogonal positive definite and
negative definite sumands, similarly $\gv'' = \gv''_1 \oplus \gv''_2$\,,
$\dim \gv'_1 = \dim \gv'_2 = u$, $\dim \gv''_1 = v$ and
$\dim \gv''_2 = w$.  So the corresponding decomposition of $\gv$ is
$\gv = \gv_1\oplus\gv_2$ where
$\gv_1 = (\gv'_1 \otimes \gv''_1) + (\gv'_2 \otimes \gv''_2)$ and
$\gv_2 = (\gv'_1 \otimes \gv''_2) + (\gv'_2 \otimes \gv''_1)$.
Thus $\dim \gv_1 = uv + uw = \dim \gv_2$\,.
That proves the second assertion on signature.
The third signature assertion follows by the same calculation.
\end{proof}

\begin{lemma}\label{sp-sp-so} We have inclusions 
{\rm (i)} $Sp(m;\R)\cdot Sp(n;\R) \subset SO(4mn,4mn)$, 
{\rm (ii)} $Sp(m;\R) \subset SO(2m,2m)$, 
{\rm (iii)} $Sp(m;\R)\cdot U(r,s) \subset SO(4mn,4mn), n=r+s$,
{\rm (iv)} $Sp(m;\R)\cdot SO(r,s) \subset SO(4mn,4mn), n=r+s$ and 
{\rm (v)} $Sp(m;\R)\cdot Sp(r,s) \subset SO(8mn,8mn), n=r+s$.
\end{lemma}
\begin{proof}
The first of these is immediate from the proof of Lemma \ref{more-prods}.  
For {\rm (ii)} view $Sp(m;\R)$ as the diagonal action on 
two paired real symplectic vector spaces of dimension $2m$, for example on
$\R^{4m} = \R^{2m} \oplus \R^{2m}$ or on $\R^{2m} \oplus \sqrt{-1}\R^{2m}$.
For {\rm (iii)}, we have the antisymmetric $\C$--bilinear $b'$ on $\C^{2m}$
and the antisymmetric $\R$--bilinear form $b''(u,v) = \Im\langle u,v\rangle$
on $\C^{r,s}$, so $b' \otimes b''$ is a symmetric bilinear form of signature
$(4mn,4mn)$ on the real vector space underlying $\C^{2m}\otimes_\R \C^{r,s}$,
and the assertion follows as in Lemma \ref{more-prods}.  Then
{\rm (iv)} follows because $SO(r,s) \subset U(r,s)$ and {\rm (v)}  follows 
because $Sp(r,s) \subset U(2r,2s)$.
\end{proof}

\section{General Heisenberg Nilmanifolds}\label{sec4}
\setcounter{equation}{0}

We recall the basic facts on commutative nilmanifolds $M_r = G_r/H_r$\,,
where $G_r = N_r \rtimes H_r$ and $N_r$ is the Heisenberg group 
$\Im\C + \C^n$\,.

\begin{proposition}\label{heis-riem}{\rm (\cite[Theorem 4.6]{BJR1990}
or see \cite[Theorem 13.2.4]{W2007})}
Let $N_r$ denote the Heisenberg group $\Im\C + \C^n$ of dimension $2n+1$,
as in {\rm Example \ref{heis}}.  Let $H_r$ be a closed connected subgroup of
$U(n)$ acting irreducibly on $\C^n$.  Then the following are equivalent.

{\rm 1. } $M_r = G_r/H_r$ is commutative,  where $G_r = N_r\rtimes H_r$\,.

{\rm 2. } The representation of $H_r$ on $\C^n$ is multiplicity free on
the ring of polynomials on $\C^n$\,.

{\rm 3. } The representation of $H_r$ on $\C^n$ is equivalent to one of
the following.

\begin{longtable}{|r|l|l|l|l|}\hline
\endhead
\hline \multicolumn{5}{r}{\textit{table continued on next page $\dots$}} \\
\endfoot
\hline
\endlastfoot

 & Group $H_r$ & Group $(H_r)_\C$ & acting on & conditions on $n$ \\ \hline
{\rm 1} & $SU(n)$ & $SL(n;\C)$ & $\C^n$ & $n \geqq 2$ \\ \hline
{\rm 2} & $U(n)$ & $GL(n;\C)$ & $\C^n$ & $n \geqq 1$ \\ \hline
{\rm 3} & $Sp(m)$ & $Sp(m;\C)$ & $\C^n$ & $n = 2m$ \\ \hline
{\rm 4} & $U(1) \cdot Sp(m)$ & $\C^* \times Sp(m;\C)$ & $\C^n$ & $n = 2m$ \\ \hline
{\rm 5} & $U(1) \cdot SO(n)$ & $\C^* \times SO(n;\C)$ & $\C^n$ & $n \geqq 2$
         \\ \hline
{\rm 6} & $U(m)$ & $GL(m;\C)$ & $S^2(\C^m)$ & $m \geqq 2, \
        n = \tfrac{1}{2}m(m+1)$  \\ \hline
{\rm 7} & $SU(m)$ & $SL(m;\C)$ & $\Lambda^2(\C^m)$ & $m$ odd,
        $n = \tfrac{1}{2}m(m-1)$  \\ \hline
{\rm 8} & $U(m)$ & $GL(m;\C)$ & $\Lambda^2(\C^m)$ &
        $n = \tfrac{1}{2}m(m-1)$  \\ \hline
{\rm 9} & $SU(\ell) \cdot SU(m)$ & $SL(\ell;\C) \times SL(m;\C)$ &
        $\C^\ell \otimes \C^m$ & $n = \ell m, \ \ell \ne m$ \\ \hline
{\rm 10} & $U(\ell) \cdot SU(m)$ & $GL(\ell;\C) \times SL(m;\C)$ &
        $\C^\ell \otimes \C^m$ & $n = \ell m$ \\ \hline
{\rm 11} & $U(2) \cdot Sp(m)$ & $GL(2;\C) \times Sp(m;\C)$ &
        $\C^2 \otimes \C^{2m}$ & $n = 4m$ \\ \hline
{\rm 12} & $SU(3) \cdot Sp(m)$ & $SL(3;\C) \times Sp(m;\C)$ &
        $\C^3 \otimes \C^{2m}$ & $n = 6m$ \\ \hline
{\rm 13} & $U(3) \cdot Sp(m)$ & $GL(3;\C) \times Sp(m;\C)$ &
        $\C^3 \otimes \C^{2m}$ & $n = 6m$ \\ \hline
{\rm 14} & $U(4) \cdot Sp(4)$ & $GL(4;\C) \times Sp(4;\C)$ &
        $\C^4 \otimes \C^8$ & $n = 32$ \\ \hline
{\rm 15} & $SU(m) \cdot Sp(4)$ & $SL(m;\C) \times Sp(4;\C)$ &
        $\C^m \otimes \C^8$ & $n = 8m, \ m \geqq 3$ \\ \hline
{\rm 16} & $U(m) \cdot Sp(4)$ & $GL(m;\C) \times Sp(4;\C)$ &
        $\C^m \otimes \C^8$ & $n = 8m, \ m \geqq 3$ \\ \hline
{\rm 17} & $U(1) \cdot Spin(7)$ & $\C^* \times Spin(7;\C)$ & $\C^8$ &
        $n = 8$ \\ \hline
{\rm 18} & $U(1) \cdot Spin(9)$ & $\C^* \times Spin(9;\C)$ & $\C^{16}$ &
        $n = 16$ \\ \hline
{\rm 19} & $Spin(10)$ & $Spin(10;\C)$ & $\C^{16}$ & $n = 16$ \\ \hline
{\rm 20} & $U(1) \cdot Spin(10)$ & $\C^* \times Spin(10;\C)$ &
        $\C^{16}$ & $n = 16$ \\ \hline
{\rm 21} & $U(1) \cdot G_2$ & $\C^* \times G_{2,\C}$ & $\C^7$ & $n=7$ \\ \hline
{\rm 22} & $U(1) \cdot E_6$ & $\C^* \times E_{6,\C}$ & $\C^{27}$ & $n=27$ \\ \hline
\end{longtable}
\end{proposition}
\noindent
In each case, $G_r/H_r$ is a weakly symmetric Riemannian manifolds; see
\cite[Theorem 15.4.7]{W2007}.

Now consider the corresponding real form families.  Following
Theorems \ref{spec-rad} and \ref{extends} we need only enumerate the real
forms $H$ of the groups $(H_r)_\C$ listed in Proposition \ref{heis-riem}.
All of them are weakly symmetric.  The only ones of Lorentz signature come from
the Riemannian cases ($H = H_r$ compact) by changing the sign of the
metric on the center $\gz$ of $\gn$.
In all cases $H$ acts trivially on $\gz$, which has dimension $1$, and
$\gv = \gv_r$ with the action of $H$ given by the restriction of the
action of $(H_r)_\C$\,.  Note that the action of $H$ on $\gv$ is irreducible
except in a few cases, such as $H = SL(n;\R)$ in Case 1, where
$\gv = \R^n \oplus \R^n$ under the action of $H$.  

In general we need and
use the tools from Sections \ref{sec2} and \ref{sec3}.
The discussion of $U(1)$ factors shows that many potential cases do not
occur.  The discussions of linear groups and signature of product groups
also eliminate many potential cases.  

Later, in Section \ref{sec5}, we will consider real form families in the 
non--Heisenberg cases.  In view of the length of the classification in 
the Heisenberg cases, we will limit our considerations in the 
non--Heisenberg setting to cases where $H_r$ is maximal in the 
following sense.  If $G_r' = N_r \ltimes H'_r$ with $G'_r/H'_r$ weakly 
symmetric and  $H_r \subset H'_r$, then $H_r = H_r$ (and so $G_r = G'_r$). 

We run through the real form families corresponding to the entries of 
the table in Proposition \ref{heis-riem}.  We use the notation $k + \ell = m$ 
and $r+s=n$ where applicable, and if we write e.g. $m/2$ for some case,
usually $GL(m/2;\H)$, then it is assumed that $m$ is even for that case.
The notation $\{L_1, \dots , L_p\}$ means any one of the $L_i$\,,
as in $\{1, U(1), \R^+\}\cdot H'$.  Our convention on possible
invariant signatures is that $(a,b)$ represents both possibilities $(a,b)$
and $(b,a)$, that $(a,b)\oplus (c,d)$ represents all four possibilities
$(a+c,b+d), (a+d,b+c), (b+d,a+c) \text{ and } (b+c,a+d)$, etc.

\underline{Cases 1 and 2.}  Lemma \ref{no-glH} shows $H \ne SL(n/2;\H)$ and
$H \ne GL(n/2;\H)$.  The signatures are obvious.

\underline{Case 3.} Lemma \ref{no-spnr} shows $H \ne Sp(m;\R)$.  
The signatures are obvious.

\underline{Case 4.}  Lemma \ref{ext2} covers the other possibilities.
The signatures come from Lemma \ref{ext3} and Remark \ref{some-prods-maybe}.

\underline{Case 5.}  The argument of Lemma \ref{no-so*} combines with the 
adjustment described in Remark \ref{some-prods-maybe} to cover the case
$H = U(1)\cdot SO^*(n)$, $n$ even.

\underline{Cases 6, 7 and 8.}  Remark \ref{s2-l2-ok} shows that the linear
groups do occur here.  The signatures for the general linear groups
come from Lemma \ref{ext3}, and they follow for the special linear groups.

\underline{Cases 9 through 16.}  The only difficulties here are with the
linear groups and the groups $Sp(m;\R)$, and for those we apply Remark 
\ref{some-prods} and Lemma \ref{more-prods}.

\underline{Cases 17 through 20.}  Here we apply Lemma \ref{spin},
Remark \ref{some-prods-ok}, and the fact that the centers
$Z(Spin(7))\cong\Z_2 \cong Z(Spin(9))$ and $Z(Spin(10)) \cong \Z_4$\,, 

\underline{Case 21.} This uses the classification of real forms of $G_{2,\C}$
and the fact that the compact simply connected $G_2$ is centerless and has 
no outer automorphisms.

\underline{Case 22.} This uses the classification of real forms of $E_{6,\C}$
and the fact that the compact simply connected $E_6$ has center $\Z_3$
and that the outer automorphisms of $E_6$ act on that center by
$z \mapsto z^{-1}$.

Omitting the obvious cases $(G_r)_\C/(H_r)_\C$\,, where $\gv$ and $\gz$ have
signatures of the form $(n,n)$ and $(1,1)$, now

\addtocounter{equation}{1}
{\footnotesize
\begin{longtable}{|r|l|l|l|} 
\caption*{\bf {\normalsize Table} \thetable \quad {\normalsize 
        Irreducible Commutative Heisenberg Nilmanifolds
        $(N_r\rtimes H_r)/H_r$}} \label{heis-comm} \\
\hline
 & Group $H$ & $\gv \text{ and signature}(\gv)$&
                $\gz\text{ and signature}(\gz)$  \\ \hline
\hline
\hline
\endfirsthead
\multicolumn{4}{l}{{\normalsize \textit{Table \thetable\, continued from 
        previous page $ \dots$}}} \\
\hline
 & Group $H$ & $\gv \text{ and signature}(\gv)$&
                $\gz\text{ and signature}(\gz)$  \\ \hline
\hline
\endhead
\hline \multicolumn{4}{r}{{\normalsize \textit{$\dots$ Table \thetable\, 
        continued on next page}}} \\
\endfoot
\hline
\endlastfoot
\hline
{\rm 1}  & $SU(n), n \geqq 2$ & $\C^n,\,\,\, (2n,0)$ & $\Im\C,\,\,\, (1,0)$ 
	\\ \hline
   & $SU(r,s)$  &  $\C^{r,s}, \,\,\, (2r,2s)$ & $\Im\C,\,\,\, (1,0)$
	\\ \hhline{|~|-|-|-|}
   & $SL(n;\R)$ &  $\R^{n,n},\,\,\,  (n,n)$   & $\Im\C,\,\,\, (1,0)$
	\\ \hline \hline
{\rm 2}  & $U(n), n \geqq 1$ &$\C^n,\,\,\, (2n,0)$ & $\Im\C,\,\,\, (1,0)$ 
	\\ \hline
   & $U(r,s)$ &  $\C^{r,s}, \,\,\, (2r,2s)$ & $\Im\C,\,\,\, (1,0)$
        \\ \hhline{|~|-|-|-|}
   & $GL(n;\R)$ &  $\R^{n,n},\,\,\,  (n,n)$   & $\Im\C,\,\,\, (1,0)$
	\\ \hline\hline
{\rm 3}  & $Sp(m)$ & $\C^{2m},\,\,\, (4m,0)$ & $\Im\C,\,\,\, (1,0)$ 
	\\ \hline
   & $Sp(k,\ell)$ & $\C^{2k,2\ell},\,\,\, (4k,4\ell)$ & $\Im\C,\,\,\, (1,0)$   
	\\ \hline\hline
{\rm 4} & $U(1) \cdot Sp(m)$ & $\C^{2m},\,\,\, (4m,0)$ & $\Im\C,\,\,\, (1,0)$
	\\ \hline
   & $\{U(1),\R^+\}\cdot Sp(k,\ell)$ & $\C^{2k,2\ell},\,\,\, (4k,4\ell)$ 
	& $\Im\C,\,\,\, (1,0)$   \\ \hhline{|~|-|-|-|}
   & $U(1) \cdot Sp(m;\R)$ & $\C^{m,m},\,\,\, (2m,2m)$ & $\Im\C,\,\,\, (1,0)$ 
	\\ \hline\hline
{\rm 5} & $SO(2) \cdot SO(n), n\geqq 2$ & $\R^{2\times n},\,\,\, (2n,0)$ 
	&  $\Im\C,\,\,\, (1,0)$ \\ \hline
   & $SO(2)\cdot SO(r,s)$ & $\R^{2\times(r,s)},\,\,\, (2r,2s)$
	&  $\Im\C,\,\,\, (1,0)$ \\ \hhline{|~|-|-|-|}
   & $SO(1,1)\cdot SO(r,s)$ & $\R^{(1,1)\times(r,s)},\,\,\, (n,n)$
	&  $\Im\C,\,\,\, (1,0)$ \\ \hhline{|~|-|-|-|}
   & $U(1)\cdot SO^*(n), n$ even & $\C^n \simeq \R^{n,n},\,\,\, (n,n)$ & $\Im\C,\,\,\, (1,0)$ 
	\\ \hline\hline
{\rm 6} & $U(m), m\geqq 2$ & $S^2_\C(\C^m),\,\,\, (m^2+m,0)$ & $\Im\C,\,\,\, (1,0)$
	\\ \hline
   & $U(k,\ell)$ & $S_\C^2(\C^{k,\ell}),\,\,\, (k^2+k+\ell^2+\ell,2k\ell))$ 
	& $\Im\C,\,\,\, (1,0)$ \\ \hhline{|~|-|-|-|}
   & $GL(m;\R)$ & $S_\C^2(\C^m) \simeq \R^{1,1}\otimes S^2_\R(\R^m),\,\,\, (\tfrac{m^2+m}{2},\tfrac{m^2+m}{2})$ 
	& $\Im\C,\,\,\, (1,0)$ \\ \hhline{|~|-|-|-|}
   & $GL(m/2;\H)$ &  $S_\C^2(\C^m) \simeq \R^{1,1}\otimes S^2_\R(\R^m),\,\,\, (\tfrac{m^2+m}{2},\tfrac{m^2+m}{2})$
	& $\Im\C,\,\,\, (1,0)$
         \\ \hline\hline
{\rm 7} & $SU(m), m$ odd & $\Lambda_\C^2(\C^m),\,\,\, (m^2-m,0)$ 
	& $\Im\C,\,\,\, (1,0)$ \\ \hline
   & $SU(k,\ell)$ & $\Lambda_\C^2(\C^{k,\ell}),\,\,\,(k^2-k+\ell^2-\ell,2k\ell)$
	& $\Im\C,\,\,\, (1,0)$ \\ \hhline{|~|-|-|-|}
   & $SL(m;\R)$ & $\Lambda_\C^2(\C^m)\simeq \R^{1,1}\otimes\Lambda_\R(\R^m),\,\,\,(\tfrac{m^2-m}{2},\tfrac{m^2-m}{2})$
         & $\Im\C,\,\,\, (1,0)$ \\ \hline\hline
{\rm 8} & $U(m)$ &  $\Lambda_\C^2(\C^m),\,\,\, (m^2-m,0)$ 
	& $\Im\C,\,\,\, (1,0)$ \\ \hline
   & $U(k,\ell)$ & $\Lambda^2_\C(\C^{k,\ell}),\,\,\,(k^2-k+\ell^2-\ell,2k\ell))$
        & $\Im\C,\,\,\, (1,0)$ \\ \hhline{|~|-|-|-|}
   & $GL(m;\R)$ & $\Lambda_\C^2(\C^m)\simeq \R^{1,1}\otimes\Lambda_\R(\R^m),\,\,\,(\tfrac{m^2-m}{2},\tfrac{m^2-m}{2})$
        & $\Im\C,\,\,\, (1,0)$ \\ \hhline{|~|-|-|-|}
   & $GL(m/2;\H)$ & $\Lambda_\C^2(\C^m)\simeq \R^{1,1}\otimes\Lambda_\R(\R^m),\,\,\, 
	(\tfrac{m^2-m}{2},\tfrac{m^2-m}{2})$ & $\Im\C,\,\,\, (1,0)$
         \\ \hline\hline
{\rm 9} & $SU(m) \cdot SU(n)$ & $\C^{m \times n},\,\,\, (2mn,0)$ 
	& $\Im\C,\,\,\, (1,0)$ \\ \hline
   & $SU(k,\ell) \cdot SU(r,s)$ & $\C^{(k,\ell) \times (r,s)},\,\,\,
	(2kr+2\ell s, 2ks + 2\ell r)$ & $\Im\C,\,\,\,(1,0)$\\ \hhline{|~|-|-|-|}
   & $SL(m;\R)\cdot SL(n;\R)$ & $\R^{m \times n} \oplus \R^{m \times n},
	\,\,\, (mn,mn)$ & $\Im\C,\,\,\, (1,0)$ \\ \hhline{|~|-|-|-|}
   & $SL(\tfrac{m}{2};\H) \cdot SL(\tfrac{n}{2};\H)$ & $\C^{m \times n},\,\,\, 
	(mn,mn)$ & $\Im\C,\,\,\, (1,0)$ \\ \hline\hline
{\rm 10} & $S(U(m) \cdot U(n))$ & $\C^{m \times n},\,\,\, (2mn,0)$ 
	& $\Im\C,\,\,\, (1,0)$ \\ \hline
   & $S(U(k,\ell) \cdot U(r,s))$ & $\C^{(k,\ell) \times (r,s)},\,\,\,
        (2kr+2\ell s, 2ks + 2\ell r)$ & $\Im\C,\,\,\,(1,0)$\\ \hhline{|~|-|-|-|}
   & $S(GL(m;\R)\cdot GL(n;\R))$ & $\R^{m \times n} \oplus \R^{m \times n},
        \,\,\, (mn,mn)$ & $\Im\C,\,\,\, (1,0)$ \\ \hhline{|~|-|-|-|}
   & $S(GL(\tfrac{m}{2};\H) \cdot GL(\tfrac{n}{2};\H))$ 
	& $\C^{m \times n},\,\,\, 
        (mn,mn)$ & $\Im\C,\,\,\, (1,0)$ \\ \hline\hline
{\rm 11} & $U(2) \cdot Sp(m)$ & $\C^2 \otimes_\C \C^{2m},\,\,\, (8m,0)$ 
	& $\Im\C,\,\,\, (1,0)$ \\ \hline
   & $U(a,b) \cdot Sp(k,\ell), \begin{smallmatrix} a+b=2\\k+\ell = m
	\end{smallmatrix}$ & $\C^{a,b} \otimes_\C \C^{2k,2\ell},\,\,\,
	(4ak+4b\ell,4a\ell +4bk)$ & $\Im\C,\,\,\,(1,0)$\\ \hhline{|~|-|-|-|}
   & $U(a,b)\cdot Sp(m;\R), a+b=2$ & $\C^{a,b}\otimes_\C \C^{2m},\,\,\,
	(4m,4m)$ & $\Im\C,\,\,\,(1,0)$\\ \hhline{|~|-|-|-|}
   & $GL(2;\R) \cdot Sp(k,\ell), k+\ell = m$ & $\C^2 \otimes_\C \C^{2k,2\ell}
	\,\,\, (4m,4m)$ & $\Im\C,\,\,\,(1,0)$ \\ \hhline{|~|-|-|-|}
   & $GL(1;\H) \cdot Sp(m;\R)$ & $\C^2 \otimes_\C \C^{2m},\,\,\, (4m,4m)$
        & $\Im\C,\,\,\, (1,0)$ \\ \hline\hline
{\rm 12} & $SU(3) \cdot Sp(m)$ & $\C^3 \otimes_\C \C^{2m},\,\,\, (12m,0)$ 
	& $\Im\C,\,\,\, (1,0)$ \\ \hline
   & $SU(a,b) \cdot Sp(k,\ell), \begin{smallmatrix} a+b=3\\k+\ell = m
        \end{smallmatrix}$ & $\C^{a,b} \otimes_\C \C^{2k,2\ell},\,\,\,
        (4ak+4b\ell,4a\ell +4bk)$ & $\Im\C,\,\,\, (1,0)$ \\ \hhline{|~|-|-|-|}
   & $SL(3;\R) \cdot Sp(k,\ell), k+\ell = m$ &
         $\C^3 \otimes_\C \C^{2k,2\ell},\,\,\, (6m,6m)$  
	& $\Im\C,\,\,\, (1,0)$ \\ \hline\hline
{\rm 13} & $U(3) \cdot Sp(m)$ & $\C^3 \otimes_\C \C^{2m},\,\,\, (12m,0)$ 
	& $\Im\C,\,\,\, (1,0)$ \\ \hline
   & $U(a,b) \cdot Sp(k,\ell), \begin{smallmatrix} a+b=3\\k+\ell = m
        \end{smallmatrix}$ & $\C^{a,b} \otimes_\C \C^{2k,2\ell},\,\,\,
        (4ak+4b\ell,4a\ell +4bk)$ & $\Im\C,\,\,\, (1,0)$ \\ \hhline{|~|-|-|-|}
   & $U(a,b) \cdot Sp(m;\R), a+b=3$ & $\C^{a,b} \otimes_\C \C^{2m},\,\,\,
	(6m,6m)$ & $\Im\C,\,\,\, (1,0)$ \\ \hhline{|~|-|-|-|}
   & $GL(3;\R) \cdot Sp(k,\ell), k+\ell = m$ & $\C^3 \otimes_\C  \C^{2k,2\ell},
	\,\,\, (6m,6m)$ & $\Im\C,\,\,\, (1,0)$  \\ \hline\hline
{\rm 14} & $U(4) \cdot Sp(4)$ & $\C^4 \otimes_\C \C^8,\,\,\, (64,0)$ 
	& $\Im\C,\,\,\, (1,0)$ \\ \hline
   & $U(a,b) \cdot Sp(k,\ell), \begin{smallmatrix} a+b=4\\k+\ell = 4
        \end{smallmatrix}$ & $\C^{a,b} \otimes_\C \C^{2k,2\ell},\,\,\,
        (4ak+4b\ell,4a\ell +4bk)$ & $\Im\C,\,\,\, (1,0)$ \\ \hhline{|~|-|-|-|}
   & $U(a,b) \cdot Sp(4;\R), a+b=4$ & $\C^{a,b} \otimes_\C \C^8,\,\,\,
        (32,32)$ & $\Im\C,\,\,\, (1,0)$ \\ \hhline{|~|-|-|-|}
   & $GL(4;\R) \cdot Sp(k,\ell), k+\ell = 4$ & $\C^4 \otimes_\C  \C^{2k,2\ell},
        \,\,\, (32,32)$ & $\Im\C,\,\,\, (1,0)$ \\ \hhline{|~|-|-|-|}
   & $GL(2;\H) \cdot Sp(4;\R)$ & $\C^4 \otimes_\C \C^8,\,\,\,
        (32,32)$ & $\Im\C,\,\,\, (1,0)$ \\ \hline\hline
{\rm 15} & $SU(m) \cdot Sp(4), m \geqq 3$ & $\C^m \otimes_\C \C^8,\,\,\,
	(16m,0)$ & $\Im\C,\,\,\, (1,0)$ \\ \hline
  & $SU(k,\ell) \cdot Sp(r,s), \begin{smallmatrix} k+\ell = m \\ r+s=4
        \end{smallmatrix}$ & $\C^{k,\ell} \otimes_\C \C^{2r,2s},\,\,\,
        (4kr+4\ell s,4ks +4\ell k)$ & $\Im\C,\,\,\, (1,0)$\\ \hhline{|~|-|-|-|}
  & $SL(m;\R) \cdot Sp(r,s), r+s=4$ & $\C^m \otimes_\C \C^{2r,2s},\,\,\,
	(8m,8m)$ & $\Im\C,\,\,\, (1,0)$\\ \hhline{|~|-|-|-|}
  & $SL(m/2;\H) \cdot Sp(4;\R)$ & $\C^m \otimes_\C \C^8,\,\,\,
        (8m,8m)$ & $\Im\C,\,\,\, (1,0)$ \\ \hline\hline
{\rm 16} & $U(m)\cdot Sp(4), m \geqq 3$ & $\C^m\otimes_\C \C^8,\,\,\, (16m,0)$ 
	& $\Im\C,\,\,\, (1,0)$ \\ \hline
   & $U(k,\ell) \cdot Sp(r,s), \begin{smallmatrix} k+\ell = m\\ r+s=4
        \end{smallmatrix}$ & $\C^{k,\ell} \otimes_\C \C^{2r,2s},\,\,\,
        (4kr+4\ell s,4ks +4\ell r)$ & $\Im\C,\,\,\, (1,0)$ \\ \hhline{|~|-|-|-|}
   & $U(k,\ell) \cdot Sp(4;\R), k+\ell =m$ & $\C^{k,\ell}\otimes_\C \C^8,\,\,\,
        (8m,8m)$ & $\Im\C,\,\,\, (1,0)$ \\ \hhline{|~|-|-|-|}
   & $GL(m;\R) \cdot Sp(r,s), r+s = 4$ & $\C^m \otimes_\C \C^{2r,2s},
        \,\,\, (8m,8m)$ & $\Im\C,\,\,\, (1,0)$ \\ \hhline{|~|-|-|-|}
   & $GL(m/2;\H) \cdot Sp(4;\R)$ & $\C^m \otimes_\C \C^8,\,\,\,
        (8m,8m)$ & $\Im\C,\,\,\, (1,0)$ \\ \hline\hline
{\rm 17} & $U(1) \cdot Spin(7)$ & $\C^8,\,\,\, (16,0)$ 
	& $\Im\C,\,\,\, (1,0)$ \\ \hline
   & $U(1) \cdot Spin(6,1)$ & $\C^{6,2},\,\,\, (12,4)$
	& $\Im\C,\,\,\, (1,0)$ \\ \hhline{~|-|-|-}
   & $U(1) \cdot Spin(5,2)$ & $\C^{6,2},\,\,\, (12,4)$
        & $\Im\C,\,\,\, (1,0)$ \\ \hhline{~|-|-|-}
   & $U(1) \cdot Spin(4,3)$ & $\C^{4,4},\,\,\, (8,8)$
        & $\Im\C,\,\,\, (1,0)$ \\ \hhline{~|-|-|-}
   & $\R^+ \cdot Spin(r,s), r+s = 7$ & $\R^{8,8},\,\,\, (8,8)$
	& $\Im\C,\,\,\, (1,0)$ \\ \hline\hline
{\rm 18} & $U(1) \cdot Spin(9)$ & $\C \otimes_\R \R^{16},\,\,\, (32,0)$ 
	& $\Im\C,\,\,\, (1,0)$ \\ \hline
   & $U(1)\cdot Spin(r,s), r+s=9$ & $\C^{8,8},\,\, (16,16)$
        & $\Im\C,\,\,\, (1,0)$ \\ \hhline{~|-|-|-}
   & $\R^+\cdot Spin(r,s), r+s=9$ & $\C^{8,8},\,\, (16,16)$
	& $\Im\C,\,\,\, (1,0)$ \\ \hline\hline
{\rm 19} & $Spin(10)$ & $\C^{16},\,\,\, (32,0)$ & $\R,\,\,\, (1,0)$ \\ \hline
   & $Spin(9,1)$ & $\R^{16,16},\,\, (16,16)$
	& $\Im\C,\,\,\, (1,0)$ \\ \hhline{~|-|-|-}
   & $Spin(8,2)$ & $\C^{8,8},\,\, (16,16)
        $ & $\Im\C,\,\, (1,0)$ \\
        \hhline{|~|-|-|-|}
   & $Spin(7,3)$ & $\H^{4,4}$,\,\, (16,16)
        & $\Im\C,\,\, (1,0)$ \\
        \hhline{|~|-|-|-|}
   & $Spin(6,4)$ & $\C^{8,8},\,\, (16,16)$
        & $\Im\C,\,\, (1,0)$ \\
        \hhline{|~|-|-|-|}
   & $Spin(5,5)$ & $\R^{16,16},\,\, (16,16)$
        & $\R,\,\, (0,1)$ \\
	 \hline\hline
{\rm 20} & $U(1) \cdot Spin(10)$ & $\C^{16},\,\,\, (32,0)$ 
	& $\Im\C,\,\,\, (1,0)$ \\ \hline
   & $\R^+\cdot Spin(9,1)$ & $\R^{16,16},\,\, (16,16)$ 
	& $\R,\,\, (0,1)$ \\ \hhline{~|-|-|-}
   & $U(1)\cdot Spin(8,2)$ & $\C^{8,8},\,\, (16,16)
        $ & $\Im\C,\,\, (1,0)$ \\
        \hhline{|~|-|-|-|}
   & $\R^+\cdot Spin(7,3)$ & $\H^{4,4}$,\,\, (16,16)
        & $\Im\C,\,\, (1,0)$ \\
        \hhline{|~|-|-|-|}
   & $U(1)\cdot Spin(6,4)$ & $\C^{8,8},\,\, (16,16)$
        & $\Im\C,\,\, (1,0)$ \\
        \hhline{|~|-|-|-|}
   & $\R^+\cdot Spin(5,5)$ & $\R^{16,16},\,\, (16,16)$
        & $\R,\,\, (0,1)$ \\
        \hhline{|~|-|-|-|}
   & $U(1)\cdot Spin^*(10)$ & $\H^{4,4},\,\, (16,16)$
        & $\Im\C,\,\, (1,0)$ \\
	\hline\hline
{\rm 21} & $U(1) \cdot G_2$ & $\C^7,\,\,\, (14,0)$ & $\Re\O,\,\,\, (1,0)$ 
	\\ \hline
   & $U(1)\cdot G_{2,A_1A_1}$ & $\C^{3,4},\,\, (6,8)$
        & $\Re\O_{sp},\,\, (1,0)$ \\
        \hhline{|~|-|-|-|}
   & $\R^+\cdot G_2$ & $\R^{1,1} \otimes_\R \R^7,\,\, (7,7)$
        & $\Re\O,\,\, (0,1)$ \\
        \hhline{|~|-|-|-|}
   & $\R^+\cdot G_{2,A_1A_1}$ & $\R^{1,1} \otimes_\R \R^{3,4},\,\, (7,7)$
        & $\Re\O_{sp},\,\, (1,0)$ \\
	\hline\hline
{\rm 22} & $U(1) \cdot E_6$ & $\C^{27},\,\,\, (54,0)$ & 
	$\Im\C,\,\,\, (1,0)$ \\ \hline
   & $U(1)\cdot E_{6,A_5A_1}$ & $\C^{15,12},\,\, (30,24)$
        & $\Im\C,\,\, (1,0)$ \\
        \hhline{|~|-|-|-|}
   & $U(1)\cdot E_{6,D_5T_1}$ & $\C^{16,11},\,\, (32,22)$
        & $\Im\C,\,\, (1,0)$ \\
        \hhline{|~|-|-|-|}
   & $\R^+\cdot E_{6,C_4}$ & $\R^{1,1}\otimes_\R \R^{27},\,\, (27,27)$
        & $\R,\,\, (0,1)$ \\
        \hhline{|~|-|-|-|}
   & $\R^+\cdot E_{6,F_4}$ &  $\R^{1,1}\otimes_\R \R^{26,1},\,\, (27,27)$
        & $\R,\,\, (0,1)$ \\
\hline
\end{longtable}
}

Certain signatures of pseudo--Riemannian metrics from Table \ref{heis-comm} 
are particularly interesting.  The Riemannian ones, of course, are just the
$G_r/H_r$, in other words those where $H$ is compact.  But every such
$G_r/H_r$ also has an invariant Lorentz metric, from the invariant 
symmetric bilinear form on $\gn_r$ that is positive definite on $\gv_r$ and
negative definite on the (one dimensional) center $\gz_r$\,.  But
inspection of Table \ref{heis-comm} shows that there are a few others,
where $\gv$ has an invariant bilinear form of Lorentz signature, say
$(d,1)$, so that $\gn$ has an invariant bilinear form of Lorentz signature
$(d+1,1)$.  For each of those, $d = 1$ and $H \cong \R^+$, so $\gn$ is
the $3$--dimensional Heisenberg algebra $\Bigl ( \begin{smallmatrix} 
0&x&z\\ 0&0&y \\0&0&0 \end{smallmatrix}\Bigr )$ where $\R^+$ acts by
$t : \Bigl ( \begin{smallmatrix} 
0&x&z\\ 0&0&y \\0&0&0 \end{smallmatrix}\Bigr ) \to
\Bigl ( \begin{smallmatrix} 
0&tx&z\\ 0&0& t^{-1}y \\0&0&0 \end{smallmatrix}\Bigr )$ and the metric
has signature $(2,1)$. 

The trans-Lorentz signature is more interesting.  Running through the 
table we see that the only cases there are the following.

\begin{proposition}\label{heis-tLorentz}
The trans--Lorentz cases in {\rm Table \ref{heis-comm}}, signature of the form 
$(p-2,2)$, all are weakly symmetric.  They are

Case 1.  $H = SU(n-1,1)$ where $G/H$ has a $G$--invariant metric of
signature $(2n-1,2)$, and $H = SL(2;\R)$ where $G/H$ has a $G$--invariant metric
of signature $(3,2)$.

Case 2.  $H = U(n-1,1)$ where $G/H$ has a $G$--invariant metric of
signature $(2n-1,2)$, and $H = GL(2;\R)$ where $G/H$ has a $G$--invariant metric
of signature $(3,2)$.

Case 4.  $H = U(1)\cdot Sp(1;\R)$ where $G/H$ has a $G$--invariant metric of
signature $(3,2)$.

Case 5.  $H = SO(2)\cdot SO(n-1,1)$ where $G/H$ has a $G$--invariant metric of
signature $(2n-1,2)$.  

Case 6.  $H = U(1,1)$ where $G/H$ has a $G$--invariant metric of
signature $(5,2)$.

Case 7.  $H = SU(2,1)$ where $G/H$ has a $G$--invariant metric of
signature $(5,2)$.

Case 8.  $H = U(2)$ and $H = U(1,1)$, where $G/H$ has a $G$--invariant metric of
signature $(1,2)$; $H = U(2,1)$ where $G/H$ has a $G$--invariant metric of
signature $(5,2)$; $H = GL(2;\R)$ and $H = GL(1;\H)$, where $G/H$
has a $G$--invariant metric of signature $(1,2)$. 
\end{proposition}

\section{Irreducible Commutative Nilmanifolds: Classification}\label{sec5}
\setcounter{equation}{0}

In our notation, Vinberg's classification of maximal 
irreducible commutative Riemannian nilmanifolds is

\addtocounter{equation}{1}
\begin{longtable}{|r|l|l|l|l|l|}
\caption*{\bf {\normalsize Table \thetable \quad Maximal Irreducible 
	Nilpotent Gelfand 
        Pairs $(N_r\rtimes H_r,H_r)$ }} \label{vin-table}\\
\hline
 & Group $H_r$ & $\gv_r$ & $\gz_r$ & $U(1)$ & max \\ \hline
\hline\hline
\endfirsthead
\multicolumn{6}{l}{\textit{Table \thetable\, continued from previous 
	page $\dots$}}\\
\hline
 & Group $H_r$ & $\gv_r$ & $\gz_r$ & $U(1)$ & max \\ \hline
\hline
\endhead
\hline \multicolumn{6}{r}{\textit{$\dots$ Table \thetable\, continued on next 
	page $\dots$}} \\
\endfoot
\hline
\endlastfoot
\hline
 & Group $H_r$ & $\gv_r$ & $\gz_r$ & $U(1)$ & max \\ \hline
1 & $SO(n)$ & $\R^n$ & $\Lambda\R^{n\times n} = \gs\go(n)$ &  & \\ \hline
2 & $Spin(7)$ & $\R^8 = \O$ & $\R^7 = \Im\O$  &  & \\ \hline
3 & $G_2$ & $\R^7 = \Im\O$ & $\R^7 = \Im\O$ &  & \\ \hline
4 & $U(1)\cdot SO(n)$ & $\C^n$ & $\Im\C$ & & $n\ne 4$ \\ \hline
5 & $(U(1)\cdot) SU(n)$ & $\C^n$ & $\Lambda^2\C^n \oplus\Im\C$ & 
        $n$ odd &  \\ \hline
6 & $SU(n), n$ odd & $\C^n$ & $\Lambda^2\C^n$ &  & \\ \hline
7 & $SU(n), n$ odd & $\C^n$ & $\Im\C$ &  & \\ \hline
8 & $U(n)$ & $\C^n$ & $\Im \C^{n\times n} = \gu(n)$ &  & \\ \hline
9 & $(U(1)\cdot) Sp(n)$ & $\H^n$ & $\Re \H^{n \times n}_0 \oplus \Im\H$ &  & 
        \\ \hline
10 & $U(n)$ & $S^2(\C^n)$ & $\R$ & & \\ \hline
11 & $(U(1)\cdot) SU(n), n \geqq 3$ & ${\Lambda}^2(\C^n)$ & $\R$ & $n$ even & 
        \\ \hline
12 & $U(1)\cdot Spin(7)$ & $\C^8$ & $\R^7 \oplus \R$ & & \\ \hline
13 & $U(1)\cdot Spin(9)$ & $\C^{16}$ & $\R$ & & \\ \hline
14 & $(U(1)\cdot) Spin(10)$ & $\C^{16}$ & $\R$ & & \\ \hline
15 & $U(1)\cdot G_2$ & $\C^7$ & $\R$ & & \\ \hline
16 & $U(1)\cdot E_6$ & $\C^{27}$ & $\R$ & & \\ \hline
17 & $Sp(1)\times Sp(n)$ & $\H^n$ & $\Im \H = \gs\gp(1)$ & & $n \geqq 2$ 
        \\ \hline
18 & $Sp(2)\times Sp(n)$ & $\C^{4\times 2n}$ & 
        $\Im \H^{2\times 2} = \gs\gp(2)$ & & \\ \hline
19 & $(U(1)\cdot) SU(m) \times SU(n)$ &  &  &  & \\
   & $m,n \geqq 3$ & $\C^m\otimes \C^n$ & $\R$ & $m=n$ &   \\ \hline
20 & $(U(1)\cdot) SU(2) \times SU(n)$ & $\C^2 \otimes \C^n$ & 
        $\Im \C^{2\times 2} = \gu(2)$ & $n=2$ & \\ \hline
21 & $(U(1)\cdot) Sp(2) \times SU(n)$ & $\H^2\otimes \C^n$ & $\R$ & 
        $n \leqq 4$ & $n \geqq 3$ \\ \hline
22 & $U(2)\times Sp(n)$ & $\C^2 \otimes \H^n$ & $\Im \C^{2\times 2} = \gu(2)$ &
        & \\ \hline
23 & $U(3)\times Sp(n)$ & $\C^3 \otimes \H^n$ & $\R$ & & $n \geqq 2$
        \\ \hline
\end{longtable}
All groups are real.  All spaces $(G_r/H_r, ds^2)$ are weakly symmetric
except for entry 9 with $H_r = Sp(n)$; see \cite[Theorem 15.4.10]{W2007}.  
This is due to Lauret \cite{L1998}.  For more details
see \cite[Section 15.4]{W2007}.  If a group $H_r$ 
is denoted $(U(1)\cdot)H_r'$ it can be
$H_r'$ or $U(1)\cdot H_r'$\,.  Under certain conditions the only case
is $U(1)\cdot H_r'$\; then those conditions are noted in the $U(1)$
column.  In this section we extend the considerations of 
Table \ref{vin-table} from 
commutative (including weakly
symmetric) Riemannian nilmanifolds to the pseudo--Riemannian setting.

Now we run through the corresponding real form families, omitting the
complexifications of the Riemannian forms $G_r/H_r$\,.
When we write $m/2$ it is implicit that we are in a case where $m$ is even,
and similarly $n/2$ assumes that $n$ is even.   Further
$k+\ell = m$ and $r+s = n$ where applicable.  We also use the notation
$\{L_1, \dots , L_p\}$ to mean any one of the $L_i$\,,
as in $\{\{1\}, U(1), \R^+\}\cdot H'$.  Finally, our convention on possible
invariant signatures is that $(a,b)$ represents both possibilities $(a,b)$
and $(b,a)$, that $(a,b)\oplus (c,d)$ represents all four possibilities
$(a+c,b+d), (a+d,b+c), (b+d,a+c) \text{ and } (b+c,a+d)$, etc.

\underline{Case 1.}  $H \ne SO^*(n)$ by Lemma \ref{no-so*}.  The
signature calculations are straightforward.

\underline{Case 2.}  The assertions follow from Lemma \ref{spin}.

\underline{Case 3.}  The assertions are obvious.

\underline{Case 4.}  Lemma \ref{no-so*} does not eliminate 
$H = U(1)\cdot SO^*(2m)$ because $\theta = \Ad(J)$, 
$J = \left ( \begin{smallmatrix} 0 & I\\-I & 0 \end{smallmatrix}\right )$,
extends to $\C^{2m}$ as $cJ$ where $c \in U(1)$ with $(cJ)^2 = 1$. 
However that is necessarily trivial on the $U(1)$ factor, so
$H = \R^+\cdot SO^*(2m)$ is eliminatred.  The
signature calculations are straightforward. 

\underline{Cases 5a and 5b.}  $H \ne \{S,G\}L(n/2;\H)$ by Lemma \ref{no-glH}.
The signatures for $(S)U(r,s)$ are obvious, and for $\{S,G\}L(n;\R)$  
they follow from Lemma \ref{ext3}. 

\underline{Cases 6 and 7.} The calculations are straightforward.

\underline{Case 8.}  $H \ne GL(n/2;\H)$ by Lemma \ref{no-glH}.  The 
signatures for $U(r,s)$ are obvious, and for $GL(n;\R)$  
they follow from Lemma \ref{ext3}.

\underline{Case 9.}  The $Sp(n;\R)$ entry depends on Example \ref{more-prods}.

\underline{Cases 10, 11a and 11b.}  The calulations are straightforward.
Note Remark \ref{s2-l2-ok} for $H=GL(n/2;\H)$.

\underline{Cases 12 and 13.}  The calculations are straightforward.

\underline{Case 14.}  $H = U(1)\cdot Spin^*(10)$ is admissible by lifting 
$\Ad(J)$, 
$J = \left ( \begin{smallmatrix} 0 & I\\-I & 0 \end{smallmatrix}\right )$,
from $SO(10)$ to $\theta = \Ad(c\widetilde{J})$ on $Spin(10)$, where
$c \in U(1)$ so that $c\widetilde{J}$ has square $1$ on $\gv_r$\,.
But $H \ne Spin^*(10)$ by Lemma \ref{no-so*}.

\underline{Case 15.}  The assertions are obvious.

\underline{Case 16.}  The assertions follow from the $E_6$ discussion in
Section \ref{sec3}.

\underline{Case 17.} $H \ne Sp(1)\times Sp(n;\R)$  and
$H \ne Sp(1;\R)\otimes Sp(r,s)$ as noted in
Remark \ref{some-prods}.  The signatures are obvious for 
$H = Sp(1)\times Sp(r,s)$.  They follow from Lemma \ref{more-prods}
for the other two cases.

\underline{Case 18.} $H$ cannot be $Sp(2)\times Sp(n;\R)$, 
$Sp(2;\R)\times Sp(r,s)$, nor $Sp(1,1)\times Sp(r,s)$, as noted in
Remark \ref{some-prods}.  See Lemma \ref{sp-sp-so} for the signatures
when $H = Sp(2;\R)\times Sp(n;\R)$.  The other signatures are immediate.

\underline{Case 19.} Lemma \ref{no-glH} eliminates both variations on
$\{1,\R^+\}(SL(m;\R)\times SL(n/2;\H))$.  The signatures in the other
cases are straightforward. 

\underline{Case 20.} $H$ cannot have semisimple part $SU(k,l)\times Sl(n/2;\H)$
by Remark \ref{some-prods-ok}.  The signatures are evident in the other
four cases.

\underline{Case 21, 22 and 23.} Remarks \ref{some-prods} and \ref{some-prods-ok}
eliminate most cases with a real symplectic group and some cases with a
quaternion linear group.  The signatures are computable.  

For the convenience of the reader in using the tables, 
Table \ref{max-irred} repeats some material from
Table \ref{heis-comm}. 

\addtocounter{equation}{1}
{\footnotesize
\begin{longtable}{|r|l|l|l|}
\caption*{\bf {\normalsize Table} \thetable \quad {\normalsize Maximal 
	Irreducible Commutative Nilmanifolds 
	$(N_r\rtimes H_r)/H_r$}} \label{max-irred} \\
\hline
 & Group $H$ & \begin{tabular}{l} $\gv$ \\ $\text{signature}(\gv)$\end{tabular}
	& \begin{tabular}{l} $\gz$ \\ $\text{signature}(\gz)$ \end{tabular} 
	\\ \hline \hline
\endfirsthead
\multicolumn{4}{l}{{\normalsize \textit{Table \thetable\, continued from 
	previous page $ \dots$}}} \\
\hline
 & Group $H$ & \begin{tabular}{l} $\gv$ \\ \text{signature}$(\gv)$\end{tabular}
	 & \begin{tabular}{l} $\gz$ \\ \text{signature}$(\gz)$\end{tabular}  \\ \hline
\hline
\endhead
\hline \multicolumn{4}{r}{{\normalsize \textit{$\dots$ Table \thetable\, 
	continued on next page}}} \\
\endfoot
\hline
\endlastfoot
\hline
1 & $SO(r,s)$ & \begin{tabular}{l} $\R^{r,s}$ \\ $(r,s)$\end{tabular}
	& \begin{tabular}{l} $\gs\go(r,s)$ \\ $(\tfrac{r(r-1)+s(s-1)}{2},rs)$ 
	  \end{tabular} \\ \hline
\hline \hline
2 & \begin{tabular}{l}$Spin(k,7-k)$ \\ $4 \leqq k\leqq 7$\end{tabular}
	& \begin{tabular}{l}$\R^{q,8-q}\,, q = 2[\tfrac{k+1}{2}]$\\
		$(q,8-q)$ \end{tabular}
		 & \begin{tabular}{l} $\R^{k,7-k}$\\ $(k,7-k)$  \end{tabular} 
	 \\ \hline \hline 
3 & $G_2$ & 
	$\Im\O,\,\,(7,0)$ 
	& $\Im\O,\,\,(7,0)$  
	 \\ \hhline{|~|-|-|-|}
  & $G_{2,A_1A_1}$ 
	& $\Im  \O_{sp},\,\,(3,4)$ 
	& $\Im \O_{sp},\,\,(3,4)$   \\
\hline \hline
4 & \begin{tabular}{l} $U(1)\cdot SO(r,s)$ \\  
	 $n=r+s\ne 4$\end{tabular} 
	& $\C^{r,s}$,\, $(2r,2s)$ 
	& $\Im\C$,\, $(1,0)$  \\ 
	\hhline{|~|-|-|-|}
  & $U(1)\cdot SO^*(2m)$ & $\C^{m,m},\,\,(2m,2m)$ 
	& $\Im \C,\,\,(1,0)$  \\
	\hhline{|~|-|-|-|}
  & $\R^+ \cdot SO(r,s)$ & $\R^{1,1} \otimes_\R \R^{r,s},\,\,(n,n)$
	& $\R$,\,\,(0,1)  \\
\hline \hline

5a & \begin{tabular}{l} $SU(r,s)$,\\ $n=r+s$ even \end{tabular} 
	& $\C^{r,s},\,\,(2r,2s)$ 
	& \begin{tabular}{l} $\Lambda_\R^2(\C^{r,s}) \oplus\Im\C$\\
		{\tiny $(2r^2-r+2s^2-s,4rs)$} \\
		{\tiny $\qquad \oplus (1,0)$} \end{tabular}
	 \\ \hhline{|~|-|-|-|}
   & $SL(n;\R)$ & $\R^{n,n},\,\,(n,n)$
        & \begin{tabular}{l} $\Lambda^2_\R(\R^{n,n}) \oplus\R$\\
                $(n^2-\tfrac{n}{2},n^2-\tfrac{n}{2})$ \\
		$\qquad \oplus (1,0)$\end{tabular}
         \\
\hline \hline

5b & $U(r,s),n=r+s$ & $\C^{r,s},\, (2r,2s)$ & 
	\begin{tabular}{l}  $\Lambda_\R^2(\C^{r,s}) \oplus\Im\C$\\
               {\tiny $(2r^2-r+2s^2-s,4rs)$} \\
		{\tiny  $\oplus (1,0)$} \end{tabular} 
		\\ \hhline{|~|-|-|-|}
  & $GL(n;\R)$ & $\R^{n,n},\,\, (n,n)$ 
	& \begin{tabular}{l} $\Lambda_\R^2(\R^{n,n}) \oplus\R$\\
                $(n^2-\tfrac{n}{2},n^2-\tfrac{n}{2})$ \\
		$\qquad \oplus (1,0)$\end{tabular}
	 \\
\hline\hline

6 & \begin{tabular}{l} $SU(r,s)$ \\ $r+s=n$ odd \end{tabular}
		& $\C^{r,s},\,\, (2r,2s)$ 
	& \begin{tabular}{l} $\Lambda^2_\R(\C^{r,s})$\\ 
		{\tiny $(2r^2-r+2s^2-s,4rs)$} \end{tabular}  
	\\ \hhline{|~|-|-|-|}
  & $SL(n;\R)$ & $\R^{n,n},\,\, (n,n)$ & 
	\begin{tabular}{l} $\Lambda^2_\R(\R^{n,n})$\\ 
		$(n^2-\tfrac{n}{2},n^2-\tfrac{n}{2})$\end{tabular}   \\
\hline\hline

7 & \begin{tabular}{l} $SU(r,s)$\\$r+s= n$ odd \end{tabular}
	& $\C^{r,s},\,\, (2r,2s)$ & $\Im\C,\,\, (1,0)$   \\ \hhline{|~|-|-|-|}
  & $SL(n;\R)$ & $\R^{n,n},\,\, (n,n)$ & $\R,\,\, (0,1)$   \\
\hline\hline

8 & $U(r,s)$ & $\C^{r,s},\,\, (2r,2s)$ & $\gu(r,s),\,\, (r^2+s^2,2rs)$  \\
	\hhline{|~|-|-|-|}
  & $GL(n;\R)$ & $\R^{n,n},\,\, (n,n)$ 
	& \begin{tabular}{l} $\gg\gl(n;\R)$ \\ 
		$(\tfrac{n(n-1)}{2}, \tfrac{n(n+1)}{2})$\end{tabular}  \\
\hline\hline

9 & $\{\{1\},U(1),\R^+\}\cdot Sp(r,s)$ & $\H^{r,s},\,\,(4r,4s)$ & 
      \begin{tabular}{l} {\tiny $\Re \H^{(r,s) \times (r,s)}_0 \oplus \Im\H $} 
        \\ {\tiny $(2n^2\text{-}n\text{-}4rs\text{-}1,4rs)$}\\
		{\tiny $\qquad \oplus (3,0)$} 
		\end{tabular} \\ \hhline{|~|-|-|-|}
   & $U(1)\cdot Sp(n;\R)$ & $\R^{2n,2n},\,\, (2n,2n)$   
	& \begin{tabular}{l} $\Re \H_{sp,0}^{n \times n} \oplus 
               \Im\H_{sp}$\\ $(n^2-1,n^2-n)$\\
		$\quad \oplus (2,1)$\end{tabular} \\
\hline\hline

10 & $U(r,s)$ & \begin{tabular}{l} $S_\C^2(\C^{r,s})$
		\\ {\tiny $(r(r+1) + s(s+1), 2rs)$}\end{tabular} 
	& $\Im\C,\,\, (1,0)$  \\
	\hhline{|~|-|-|-|}
  & $GL(n;\R)$ & \begin{tabular}{l} $\R^{1,1}\otimes_\R S_\R^2(\R^n)$\\
	{\tiny $(\tfrac{n(n+1)}{2},  \tfrac{n(n+1)}{2})$}\end{tabular} 
		&  $\R,\,\, (0,1)$ \\
	\hhline{|~|-|-|-|}
  & $GL(\tfrac{n}{2};\H)$ & \begin{tabular}{l} $S^2_\C(\C^n)$\\
	{\tiny $(\tfrac{n(n+1)}{2}, \tfrac{n(n+1)}{2})$}\end{tabular} 
		 & $\R,\,\, (0,1)$  \\
\hline\hline

11a & \begin{tabular}{l} $SU(r,s)$ \\ $r+s=n>3$ odd \end{tabular} 
	& \begin{tabular}{l} $\Lambda^2_\C(\C^{r,s})$ \\
		$(r^2-r+s^2-s,2rs)$\end{tabular}  
	& $\Im\C,\,\, (1,0)$  \\
	\hhline{|~|-|-|-|}
   & $SL(n;\R)$ 
	& \begin{tabular}{l} $\R^{1,1}\otimes_\R \Lambda_\R^2(\R^n)$\\
		$(\tfrac{n(n-1)}{2}, \tfrac{n(n-1)}{2}))$ \end{tabular}
	& $\R$,\,\, (0,1)  \\ 
\hline\hline        

11b & \begin{tabular}{l} $U(r,s)$ \\ $r+s=n \geqq 3$ \end{tabular}
	& \begin{tabular}{l}  ${\Lambda}^2_\C(\C^{r,s})$ \\
	  $(r^2-r+s^2-s,2rs)$\end{tabular} & $\Im\C,\,\, (1,0)$  \\
	\hhline{|~|-|-|-|}
    & $GL(n;\R)$ 
	& \begin{tabular}{l} {\tiny $\R^{1,1}\otimes_\R \Lambda_\R(\R^n)$}\\
	  {\tiny $(\tfrac{n(n-1)}{2}, \tfrac{n(n-1)}{2})$}\end{tabular}
		 & $\R$,\,\, (0,1)  \\ 
	\hhline{|~|-|-|-|}
    & $H= GL(\tfrac{n}{2};\H)$ 
	& \begin{tabular}{l} ${\Lambda}^2_\C(\C^n)$ \\
	 $(\tfrac{n(n-1)}{2}, \tfrac{n(n-1)}{2})$\end{tabular} 
	  & $\R,\,\, (0,1)$  \\ 
\hline\hline        

12 & \begin{tabular}{l} $U(1)\cdot Spin(k,7-k)$ \\
		$4\leqq k\leqq 7,\, q=2[\tfrac{k+1}{2}]$ \end{tabular}
	& \begin{tabular}{l} $\C\otimes_\R \R^{q,8-q}$ \\
		$(2q,16-2q)$ \end{tabular}
	& \begin{tabular}{l} $\R^{k,7-k} \oplus \Im\C$\\
		$(k,7-k)\oplus (1,0)$ \end{tabular} \\ \hhline{|~|-|-|-|}
  &\begin{tabular}{l} $\R^+\cdot Spin(k,7-k)$ \\
                $4\leqq k\leqq 7,\, q=2[\tfrac{k+1}{2}]$ \end{tabular}
        & \begin{tabular}{l} $\R^{1,1}\otimes_\R \R^{q,8-q}$ \\
                $(8,8)$ \end{tabular}
        & \begin{tabular}{l} $\R^{k,7-k} \oplus \R$\\
                $(k,7-k)\oplus (0,1)$ \end{tabular} \\ \hline\hline

13 & \begin{tabular}{l} $U(1)\cdot Spin(k,9-k)$\\
		$5\leqq k\leqq 9,\, q=2^{1+[\tfrac{k+3}{4}]}$ \end{tabular}
	& \begin{tabular}{l} $\C\otimes_\R \R^{q,16-q}$\\
		$(2q,32-2q)$\end{tabular}
	& $\Im\C,\,\, (1,0)$ \\ \hhline{|~|-|-|-|}
   & \begin{tabular}{l} $\R^+\cdot Spin(k,9-k)$\\
                $5\leqq k\leqq 9,\, q=2^{1+[\tfrac{k+3}{4}]}$ \end{tabular}
        & \begin{tabular}{l} $\R^{1,1}\otimes_\R \R^{q,16-q}$\\
		$(16,16)$ \end{tabular}
	& $\R,\,\, (0,1)$ \\ \hline\hline

14 & $\{\{1\} ,\, U(1)\}\cdot )Spin(10)$ & $\C^{16},\,\, (32,0)$ 
	& $\Im\C,\,\, (1,0)$ \\ \hhline{|~|-|-|-|}
   & $\{\{1\} ,\, \R^+\}\cdot Spin(9,1)$ & $\R^{16,16},\,\, (16,16)$ 
	& $\R,\,\, (0,1)$ \\
	\hhline{|~|-|-|-|}
   & $\{\{1\} ,\, U(1)\}\cdot Spin(8,2)$ & $\C^{8,8},\,\, (16,16)
	$ & $\Im\C,\,\, (1,0)$ \\
	\hhline{|~|-|-|-|}
   & $\{\{1\} ,\, \R^+\}\cdot Spin(7,3)$ & $\H^{4,4}$,\,\, (16,16) 
	& $\R,\,\, (0,1)$ \\
	\hhline{|~|-|-|-|}
   & $\{\{1\} ,\, U(1)\}\cdot Spin(6,4)$ & $\C^{8,8},\,\, (16,16)$ 
	& $\Im\C,\,\, (1,0)$ \\
	\hhline{|~|-|-|-|}
   & $\{\{1\},\,  \R^+\}\cdot Spin(5,5)$ & $\R^{16,16},\,\, (16,16)$ 
	& $\R,\,\, (0,1)$ \\
	\hhline{|~|-|-|-|}
   & $U(1)\cdot Spin^*(10)$ & $\H^{4,4},\,\, (16,16)$ 
	& $\Im\C,\,\, (1,0)$ \\
\hline\hline

15 & $U(1)\cdot G_2$ & $\C^7 = \Im\O_\C,\,\, (14,0)$ 
	& $\R = \Re\O,\,\, (0,1)$ \\ \hline
   & $U(1)\cdot G_{2,A_1A_1}$ & $\C^{3,4},\,\, (6,8)$ 
	& $\Re\O_{sp},\,\, (1,0)$ \\ 
	\hhline{|~|-|-|-|}
   & $\R^+\cdot G_2$ & $\R^{1,1} \otimes_\R \R^7,\,\, (7,7)$ 
	& $\Re\O,\,\, (0,1)$ \\ 
	\hhline{|~|-|-|-|}
   & $\R^+\cdot G_{2,A_1A_1}$ & $\R^{1,1} \otimes_\R \R^{3,4},\,\, (7,7)$ 
	& $\Re\O_{sp},\,\, (1,0)$ \\ 
\hline\hline

16 & $U(1)\cdot E_6$ & $\C^{27},\,\, (54,0)$ 
	& $\Im\C,\,\, (1,0)$ \\ \hline
   & $U(1)\cdot E_{6,A_5A_1}$ & $\C^{15,12},\,\, (30,24)$ 
	& $\Im\C,\,\, (1,0)$ \\ 
	\hhline{|~|-|-|-|}
   & $U(1)\cdot E_{6,D_5T_1}$ & $\C^{16,11},\,\, (32,22)$ 
	& $\Im\C,\,\, (1,0)$ \\ 
	\hhline{|~|-|-|-|}
   & $\R^+\cdot E_{6,C_4}$ & $\R^{1,1}\otimes_\R \R^{27},\,\, (27,27)$ 
	& $\R,\,\, (0,1)$ \\
	\hhline{|~|-|-|-|}
   & $\R^+\cdot E_{6,F_4}$ &  $\R^{1,1}\otimes_\R \R^{26,1},\,\, (27,27)$ 
	& $\R,\,\, (0,1)$ \\
\hline\hline

17 & $Sp(1)\cdot Sp(r,s), r+s\geqq 2$ & $\H^{r,s},\,\, (4r,4s)$ 
	& $\gs\gp(1),\,\, (3,0)$ \\
	\hhline{|~|-|-|-|}
   & $Sp(1;\R)\cdot Sp(n;\R)$ & $\R^{2n,2n},\,\, (2n,2n)$
	& $\gs\gp(1;\R),\,\, (1,2)$ \\
\hline\hline

18 & \begin{tabular}{l} $Sp(2)\cdot Sp(r,s)$ \\ $r+s=n \geqq 2$\end{tabular} 
	& $\C^4\otimes_\C \C^{2r,2s}$,\,  
		$(16r,16s)$ 
	& $\gs\gp(2)$,\,\, $(10,0)$  \\ 
	\hhline{|~|-|-|-|}
   & $Sp(1,1)\cdot Sp(r,s)$ & $\C^{2,2}\otimes_\C\C^{2r,2s},\,\, (8n,8n)$
	& $\gs\gp(1,1),\,\, (6,4)$ \\
	\hhline{|~|-|-|-|}
   & $H = Sp(2;\C)$ ($n=2$) & $\C^{4\times 4},\,\,(16,16)$ 
	& $\gs\gp(2),\,\, (10,0)$   \\
	\hhline{|~|-|-|-|}
   & $H=Sp(2;\R)\cdot Sp(n;\R)$ &  $\R^{8n,8n},\,\, (8n,8n)$ 
	& $\gs\gp(2;\R),\,\, (4,6)$  \\
\hline\hline
        
19 & \begin{tabular}{l} $\{\{1\} ,\, U(1)\}\cdot$\\
		$\quad (SU(k,\ell) \cdot SU(r,s))$ \\
		$m=k+\ell,n=r+s \geqq 3$ \\
	$U(1) \text{ required if } m = n$ \end{tabular}
	& \begin{tabular}{l} $\C^{(k,\ell)\times (r,s)}$\\
		$(2kr+2\ell s, 2ks + 2\ell r)$\end{tabular} 
	& $\Im \C,\,\,(1,0)$  \\ \hhline{|~|-|-|-|}
  & \begin{tabular}{l} $\R^+ \cdot SL(m;\C)$ ($m=n$) \end{tabular}
  	& \begin{tabular}{l} $\gg\gl(m;\C)$\\$(m^2,m^2)$\end{tabular} 
	& $\Im\C,\,\, (1,0)$ \\
	\hhline{|~|-|-|-|}
  & \begin{tabular}{l} $\{\{1\} ,\, \R^+\}\cdot $ \\
		$\quad (SL(m;\R)\cdot SL(n;\R))$ \end{tabular}
	& \begin{tabular}{l} $\R^{1,1}\otimes_\R \R^{m\times n}$ \\
		$(mn,mn)$ \end{tabular} & $\R,\,\, (0,1)$ \\
	\hhline{|~|-|-|-|}
  & \begin{tabular}{l} $\{\{1\} ,\, \R^+\}\cdot$ \\ 
		$\quad (SL(\tfrac{m}{2};\H) \cdot SL(\tfrac{n}{2};\H))$ 
		\end{tabular}
	&  \begin{tabular}{l} $\R^{1,1}\otimes_\R \H^{(m/2) \times (n/2)}$\\
		$(mn,mn)$ \end{tabular}
	& $\R,\,\, (0,1)$  \\
\hline\hline

20 & \begin{tabular}{l} $\{\{1\} ,\,  U(1)\}\cdot$\\ 
		$\quad (SU(k,\ell) \cdot SU(r,s))$ \\ 
		$k+\ell =2, n =r+s \geqq 2$ \\ 
		$U(1) \text{ required if } n=2$ \end{tabular} 
	& \begin{tabular}{l} $\C^{(k,\ell)\times (r,s)}$ \\
		$(2kr+2\ell s, 2ks+2\ell r)$ \end{tabular}
	& \begin{tabular}{l} $\gu(k,\ell)$ \\
		$(2k,2\ell)$\end{tabular}   \\ 
	\hhline{|~|-|-|-|}	
    & \begin{tabular}{l} $\{\{1\} ,\, \R^+\}\cdot$\\
		$\quad  (SL(1;\H)\cdot SL(n/2;\H))$ \\
		\end{tabular}
	& $\C^{2\times n},\,\, (2n,2n)$
	& $\gg\gl(1;\H),\,\, (3,1)$ \\
        \hhline{|~|-|-|-|}
    & \begin{tabular}{l} $\{\{1\},\,\R^+\}\cdot(SL(2;\R) \cdot $\\ 
		$\quad SL(n;\R))$ \\
                \end{tabular}
	& \begin{tabular}{l} $\R^{1,1}\otimes_\R \R^{2\times n}$ \\
		$(2n,2n)$ \end{tabular}
	& $\gg\gl(2;\R), \,\, (1,3)$  \\
\hline\hline

21 & \begin{tabular}{l} $\{\{1\} ,\, U(1)\}\cdot Sp(k,\ell)\cdot$\\
	$\quad SU(r,s), \begin{smallmatrix} k+\ell = 2 \\
		r+s = n \geqq 3\end{smallmatrix} $\\
	$U(1) \text{ required if } n \leqq 4$
	\end{tabular}
	& \begin{tabular}{l} $\H^{k,\ell} \otimes_\R \C^{r,s}$ \\
		$(8kr + 8\ell s, 8ks + 8\ell r)$\end{tabular}
	& $\Im\C,\,\, (1,0)$
	\\ \hhline{|~|-|-|-|} 
   & \begin{tabular}{l} $\{\{1\} ,\, \R^+\}\ \cdot Sp(k,\ell)\cdot$\\
	$\quad SL(n;\R)),\, \R^+ \text{ if } n \leqq 4$ \end{tabular}
	&\begin{tabular}{l} $\H^{k,\ell}\otimes_\R \R^{n,n}$ \\ 
		$(8n,8n)$ \end{tabular}
	& $\Im\C,\,\, (1,0)$
        \\ \hhline{|~|-|-|-|}
   & $Sp(2;\R) \cdot U(r,s))$ 
	& $\R^{4,4}\otimes_\R \C^{r,s} ,\,\, (8n,8n)$ & $\Im\C,\,\, (1,0)$     
	\\ \hhline{|~|-|-|-|}
   & \begin{tabular}{l} $\{\{1\} ,\, \R^+\}\ \cdot (Sp(2;\R)\cdot $ \\
	 $\quad SL(\tfrac{n}{2};\H),\, \R^+ \text{ if } n\leqq 4$\end{tabular} 
	& $\R^{4,4}\otimes_\R \H^{n/2},\,\, (8n,8n)$ 
	& $\R,\,\, (0,1)$     \\
\hline\hline

22 & $U(k,\ell)\cdot Sp(r,s), \begin{smallmatrix} k+\ell =2 \\r+s=n
		\end{smallmatrix}$ 
	& \begin{tabular}{l} $\C^{k,\ell} \otimes_\C \C^{2r,2s}$ \\
		$(4kr+4\ell s, 4ks+4\ell r)$ \end{tabular}
	& \begin{tabular}{l} $\gu(k,\ell)$ \\
		$(2k,4-2k)$\end{tabular}   
	\\ \hhline{|~|-|-|-|}
   & $U(k,\ell)\cdot Sp(n;\R)$ 
	& \begin{tabular}{l} $\C^{k,\ell}\otimes_\C \C^{n,n}$\\
		$(4n,4n)$ \end{tabular}
	& \begin{tabular}{l} $\gu(k,\ell)$ \\
                $(2k,4-2k)$ \end{tabular}
	\\ \hhline{|~|-|-|-|}
   & $GL(2;\R)\cdot Sp(r,s)$ 
	& $\R^2 \otimes_\R \H^{r,s},\,\, (4n,4n)$ 
	& $\gg\gl(2;\R),\,\, (1,3)$   \\ \hhline{|~|-|-|-|}
   & $GL(1;\H)\cdot Sp(n;\R)$ & $\H \otimes_\C \C^{2n},\,\, (4n,4n)$
	& $\gg\gl(1;\H),\,\, (3,1)$   \\
\hline\hline

23 & \begin{tabular}{l} $U(k,\ell)\cdot Sp(r,s)$\\
	{\tiny $k+\ell=3, n=r+s\geqq 2$} \end{tabular}
	& \begin{tabular}{l} $\C^{k,\ell} \otimes_\C \C^{2r,2s}$\\
	  $(4kr+4\ell s, 4ks+4\ell r)$ \end{tabular}
	& $\Im\C,\,\, (1,0)$   \\ \hhline{|~|-|-|-|}
   & $U(k,\ell)\cdot Sp(n;\R)$
	& $\C^{k,\ell} \otimes_\C \C^{2n},\,\, (6n,6n)$
	& $\Im\C,\,\, (1,0)$   \\ \hhline{|~|-|-|-|}
   & $GL(3;\R)\cdot Sp(r,s)$ 
	& $\C^3 \otimes_\C \H^{r,s},\,\, (6n,6n)$ & $\R,\,\, (0,1)$   \\
\hline
\end{longtable}
}
\noindent
We now extract special signatures from Table \ref{max-irred}.  In order to
avoid redundancy we consider $SO(n)$ only for $n\geqq 3$, $SU(n)$ and $U(n)$ 
only for $n\geqq 2$, and $Sp(n)$ only for $n\geqq 1$.

\begin{corollary}\label{Lorentz1}
The Lorentz cases, signature of the form $(p-1,1)$ in 
{\rm Table \ref{max-irred}}, all are weakly symmetric.  In addition
to their invariant Lorentz metrics, each has 
invariant weakly symmetric Riemannian metrics: 

Case 4.  $H=U(1)\cdot SO(n)$ with $G$--invariant metric on $G/H$ of signature 
$(2n,1)$ 

Case 5. $H=SU(n)$ and $H=U(n)$, each with $G$--invariant metric on $G/H$ of 
signature $(2n^2+n,1)$

Case 7. $H=SU(n)$ with $G$--invariant metric on $G/H$ of signature $(2n,1)$

Case 10. $H=U(n)$ with $G$--invariant metric on $G/H$ of signature $(n^2+n,1)$

Case 11. $H=SU(n)$ and $H=U(n)$, each with $G$--invariant metric on $G/H$ of 
signature $(n^2-n,1)$

Case 12.  $H=U(1)\cdot Spin(7)$ with $G$--invariant metric on $G/H$ of 
signature $(23,1)$

Case 13.  $H=U(1)\cdot Spin(9)$ with $G$--invariant metric on $G/H$ of 
signature $(32,1)$

Case 14.  $H=(U(1)\cdot)Spin(10)$ with $G$--invariant metric on $G/H$ of 
signature $(32,1)$

Case 15.  $H=U(1)\cdot G_2$ with $G$--invariant metric on $G/H$ of signature 
$(14,1)$

Case 16.  $H=U(1)\cdot E_6$ with $G$--invariant metric on $G/H$ of signature 
$(54,1)$

Case 19.  $H=(U(1)\cdot)(SU(m)\cdot SU(n))$ with $G$--invariant metric on 
$G/H$ of signature $(2mn,1)$
 
Case 21.  $H=(U(1)\cdot) (Sp(2)\cdot SU(n))$ with $G$--invariant metric on 
$G/H$ of signature $(16n,1)$
 
Case 23.  $H=U(3)\cdot Sp(n)$ with 
$G$--invariant metric on $G/H$ of signature $(12n,1)$

\end{corollary}

\begin{corollary}\label{t-Lorentz1}
The complexifications of the Lorentz cases listed in {\rm Corollary
\ref{Lorentz1}} all are of trans--Lorentz signature $(p-2,2)$.
The trans--Lorentz cases, signature of the form $(p-2,2)$ in 
{\rm Table \ref{max-irred}}, all are weakly symmetric, and they are 
given as follows.

Case 1.  $H=SO(2,1)$ with $G$--invariant metric on $G/H$ of signature $(4,2)$

Case 4.  $H=U(1)\cdot SO(n-1,1)$ with $G$--invariant metric on $G/H$ of 
signature $(2n-1,2)$

Case 7.  $H=SU(n-1,1)$ with $G$--invariant metric on $G/H$ of 
signature $(2n-1,2)$

Case 10.  $H=U(1,1)$ with $G$--invariant metric on $G/H$ of signature $(5,2)$

Case 11. $H= SU(2,1)$ and $H = U(2,1)$, each with $G$--invariant metric on 
$G/H$ of signature $(5,2)$

\end{corollary}

\section{Indecomposable Commutative Nilmanifolds}\label{sec6}
\setcounter{equation}{0}
In this section we broaden the scope of Table \ref{max-irred} from
irreducible to indecomposable commutative spaces --- subject to a few
technical conditions.  This is based on a classification of Yakimova 
(\cite{Y2005}, \cite{Y2006}; or see \cite{W2007}). It settles the case 
where $(N\rtimes H,H)$ is indecomposable, principal, maximal and 
$Sp(1)$--saturated.  

Since $G = N\rtimes H$ acts almost--effectively on $M = G/H$, the
centralizer of $N$ in $H$ is discrete, in other words the representation
of $H$ on $\gn$ has finite kernel.  (In the notation of 
\cite[Section 1.4]{Y2006} this says $H = L = L^\circ$ and $P = \{1\}$.)  
That simplifies the general definitions \cite[Definition 6]{Y2006} of 
{\em principal} and \cite[Definition 8]{Y2006} of {\em $Sp(1)$--saturated}, 
as follows.  Decompose $\gv$ as a sum $\gw_1 \oplus \dots \oplus \gw_t$ of 
irreducible $\Ad(H)$--invariant subspaces.  Then $(G,H)$ is {\bf principal}
if $Z_H^0 = Z_1 \times \dots \times Z_m$ where $Z_i \subset GL(\gw_i)$,
in other words $Z_i$ acts trivially on $\gw_j$ for $j \ne i$.
Decompose $H = Z_H^0 \times H_1 \times \dots \times H_m$ where the $H_i$ are
simple.  Suppose that whenever some $H_i$ acts nontrivially on some $\gw_j$
and $Z_H^0 \times \prod_{\ell \ne i} H_\ell$ is irreducible on $\gw_j$, it
follows that $H_i$ is trivial on $\gw_k$ for all $k \ne j$.  Then
$H_i \cong Sp(1)$ and we say that $(G,H)$ is {\bf $Sp(1)$--saturated}.
The group $Sp(1)$ will be more visible in the definition when we
extend the definition to the cases where $H \ne L$.

In the Table \ref{indecomp} below, $\gh_{n;\F}$ is the Heisenberg algebra
$\Im\F + \F^n$ of real dimension $(\dim_\R\F - 1)+ n\dim_\R\F$.  Here
$\F$ is the real, complex, quaternion or octonion algebra over $\R$, 
$\Im\F$ is its imaginary component, and
$$
\gh_{n;\F} = \Im \F + \F^n \text{ with product }
	[(z_1,v_1),(z_2,v_2)] = (\Im(v_1\cdot v_2^*),0)
$$
where the $v_i$ are row vectors and $v_2^*$ denotes the conjugate ($\F$
over $\R$) transpose of $v_2$\,.  It is the 
Lie algebra of the (slightly generalized) Heisenberg group $H_{n;\F}$\,.
Also in the table, in the listing for $\gn$
the summands in double parenthesis ((..)) are the subalgebras
$[\gw,\gw] + \gw$ where $\gw$ is an $H$--irreducible subspace of $\gv$
with $[\gw,\gw] \ne 0$, and the summands not in double parentheses are 
$H$--invariant subspaces $\gw$ with $[\gw,\gw] = 0$.  The center $\gz =
[\gn,\gn] + \gu$ where $\gu$ is the sum of those $\gw$ with $[\gw,\gw] = 0$.
Thus $\gn = \gz + \gv$ where the center $\gz$ is
the sum of $[\gn,\gn]$ with those summands listed for $\gn$ that are {\it not}
enclosed in double parenthesis ((..)).

As before, when we write $m/2$ it is assumed that 
$m$ is even, and similarly $n/2$ requires that $n$ be even.   Further
$k+\ell = m$ and $r+s = n$ where applicable.  In the signatures column we
write $\gn'$ for $[\gn,\gn]$.

\addtocounter{equation}{1}
\begin{longtable}{|r|l|l|}
\caption*{\bf \small
Table \thetable \quad Maximal Indecomposable Principal
Saturated Nilpotent Gelfand Pairs $(N\rtimes H,H)$ \\
\centerline{$N$ Nonabelian, Where the Action of $H$ on
$\gn/[\gn,\gn]$ is Reducible}} \label{indecomp} \\
\hline
  & Group $H$, Algebra $\gn$ & Signatures \\
\hline
\hline
\endfirsthead
\multicolumn{3}{l}{{\normalsize \textit{Table \thetable\, continued from
        previous page $\dots$}}} \\ \hline
  & Group $H$, Algebra $\gn$ & Signatures \\
\hline
\hline
\endhead
\hline \multicolumn{3}{r}{{\normalsize \textit{$\dots$ Table \thetable\,
        continued on next page}}} \\
\endfoot
\hline
\endlastfoot

1 & \begin{tabular}{l} $U(r,s)$ \\ $((\gh_{r+s;\C})) 
		\oplus \gs\gu(r,s)$ \end{tabular}
  & \begin{tabular}{l} $\gv: (2r,2s)$ \\ $\gu: (r^2+s^2-1,2rs)$ 
	\\ $\gn': (1,0)$ \end{tabular} \\ \hhline{|~|-|-|}
  & \begin{tabular}{l} $GL(n;\R)$ \\ $((\gh_{n;\C}))
		\oplus \gs\gl(n;\R)$ \end{tabular}
  & \begin{tabular}{l} $\gv: (n,n)$
	\\ $\gu: (n(n-1)/2, n(n+1)/2-1)$
		\\ $\gn': (0,1)$ \end{tabular} \\ \hline\hline	

2 & \begin{tabular}{l} $U(k,\ell),\, (k,\ell) = (4,0) \text{ or } (2,2)$ \\
           $((\Im\C + \Lambda^2(\C^{k,\ell}) +\C^{k,\ell})) 
		\oplus \Lambda^2(\R^{k,\ell})$ \end{tabular}
	& \begin{tabular}{l} $\gv: (2k,2\ell)$ 
		\\ $\gu: (k(k-1)/2 +\ell(\ell -1)/2,k\ell)$ 
		\\ $\gn': (1,0)\oplus ((12,0) \text{ or } (4,8))$ 
		\end{tabular} \\ \hhline{|~|-|-|}
  & \begin{tabular}{l} $GL(4;\R)$ \\ 
	  $((\R + \Lambda^2(\C^{2,2}) + \C^{2,2})) \oplus \R^{3,3}$\end{tabular}
	& \begin{tabular}{l}  $\gv: (4,4)$ 
		\\ $\gu: (3,3) \text{ and }
		\gn': (0,1)\oplus (6,6)$
		\end{tabular} \\ \hline\hline

3 & \begin{tabular}{l} $U(1)\cdot SU(r,s) \cdot U(1),\, r+s=n$ \\
        $((\gh_{n;\C}))\oplus \gh_{n(n-1)/2;\C}))$\end{tabular}
        & \begin{tabular}{l} $\gv: (2r,2s)\oplus (r^2-r+s^2-s,2rs)$ \\
		$\gu: 0$ \text{ and }
		$\gn': (1,0) \oplus (1,0)$ \end{tabular}
		\\ \hhline{|~|-|-|}
  & \begin{tabular}{l} $\R^+\cdot SL(n;\R)\cdot \R^+$ \\
        $((\gh_{n;\C}))\oplus +((\gh_{n(n-1)/2;\C}))$\end{tabular}
        & \begin{tabular}{l} $\gv: (n,n)\oplus 
			(\frac{n(n-1)}{2},\frac{n(n-1)}{2})$\\
		$\gu: 0$ \text{ and }
                $\gn': (0,1)\oplus (0,1)$\end{tabular} 
       \\ \hline\hline

4 & \begin{tabular}{l} $SU(2k,2\ell),\, (k,\ell) = (2,0) \text{ or } (1,1)$ \\
	$((\Im\C + \Re\H^{(k,\ell)\times (k,\ell)} + \C^{2k,2\ell}))$ \\
		$\qquad \oplus \R^{k(2k-1)+\ell(2\ell -1),4k\ell}$\end{tabular}
	& \begin{tabular}{l} $\gv: (4k,4\ell)$ \\
		$\gu: (k(2k-1)+\ell(2\ell -1),4k\ell)$\\
		$\gn': (1,0) \oplus ((6,0) \text{ or } (2,4))$ \end{tabular}
	\\  \hhline{|~|-|-|}
  & \begin{tabular}{l} $SL(4;\R)$ \\ $((\R + \Re\H_{sp}^{2\times 2} 
		+ \R^{4,4}))\oplus \R^{3,3} $\end{tabular}
	& \begin{tabular}{l} $\gv: (4,4)$\\
		$\gu: (3,3)$\text{ and }
		$\gn': (0,1) \oplus (4,2)$ \end{tabular}  
	\\ \hline\hline

5 & \begin{tabular}{l}$ U(k,\ell)\times U(2r,2s), 
		\begin{smallmatrix} k+\ell=2\\ r+s=2\end{smallmatrix}$
		\\ $((\gu(k,\ell) + \C^{(k,\ell)\times (2r,2s)}))$ \\ 
			$\qquad \oplus \R^{r(2r-1)+s(2s-1), 4rs}$ \end{tabular}
	& \begin{tabular}{l} $\gv: (2kr+2\ell s, 2ks+2\ell r)$ \\
		$\gu: (r(2r-1)+s(2s-1),4rs)$ \\
		$\gn': (2r,2s)$ \end{tabular}\\  \hhline{|~|-|-|}
  & \begin{tabular}{l} $GL(2;\R)\cdot  GL(4;\R)$ \\
		$((\gg\gl(2;\R) + \R^{1,1} \otimes \R^{2\times 4})) 
			\oplus \R^{3,3}$ \end{tabular}
	& \begin{tabular}{l} $\gv: (8,8)$ \\
		$\gu: (3,3)$ \text{ and }
		$\gn': (1,3)$ \end{tabular} \\ \hhline{|~|-|-|}
  & \begin{tabular}{l} $GL(1;\H)\cdot GL(2;\H)$ \\
		$((\gg\gl(1;\H) + \R^{1,1} \otimes_\R \H^2)) 
			\oplus \R^{5,1}$ \end{tabular}
	& \begin{tabular}{l} $\gv: (8,8)$ \\
                $\gu: (5,1)$ \text{ and }
                $\gn': (3,1)$ \end{tabular}
	\\ \hline\hline

6 &  \begin{tabular}{l} $S(U(2k,2\ell)\times U(r,s)),\, 
		\begin{smallmatrix} k+\ell = 2 \\ r+s=n \end{smallmatrix}$\\
		$((\gh_{4n;\C})) \oplus 
			\R^{k(2k-1)+\ell(2\ell -1),4k\ell}$ \end{tabular}
	&\begin{tabular}{l} $\gv: (8r,8s)$ \\
                $\gu: (k(2k-1)+\ell (2\ell -1),4k\ell)$ \\
                $\gn': (1,0)$ \end{tabular} \\ \hhline{|~|-|-|}

   & \begin{tabular}{l} $S(GL(4;\R) \cdot GL(n;\R))$ \\
                $((\R + \R^{1,1}\otimes \R^{4\times n}))
		\oplus \R^{3,3}$ \end{tabular}
	&\begin{tabular}{l} $\gv: (4n,4n)$ \\
		$\gu: (3,3)$ \text{ and }
		$\gn': (0,1)$  \end{tabular} \\ \hhline{|~|-|-|}
   & \begin{tabular}{l} $S(GL(2;\H) \cdot GL(n/2;\H))$ \\
		$((\R + \R^{1,1}\otimes_{\R}\H^{2\times n/2})) \oplus\R^{5,1}$
			\end{tabular}
	&\begin{tabular}{l} $\gv: (4n,4n)$ \\
                $\gu: (5,1)$ \text{ and }
                $\gn': (0,1)$  \end{tabular} \\ \hline\hline

7 & \begin{tabular}{l} $U(k,\ell)\cdot U(r,s),\, 
		\begin{smallmatrix} k+\ell = m \\ r+s=n \end{smallmatrix}$\\
		$((\gh_{mn;\C})) \oplus ((\gh_{m;\C}))$ \end{tabular}
	&\begin{tabular}{l} $\gv: (2kr+2\ell s,2ks+2\ell r)\oplus (2k,2\ell)$\\
                $\gu: 0$ \text{ and }
                $\gn': (1,0)\oplus (1,0)$  \end{tabular} \\ \hhline{|~|-|-|}

   & \begin{tabular}{l} $GL(m;\R) \cdot GL(n;\R)$ \\
               $((\R + \R^{1,1}\otimes \R^{m\times n}))
                \oplus ((\R + \R^{m,m}))$ \end{tabular}
	&\begin{tabular}{l} $\gv: (mn,mn)\oplus (m,m)$\\
                $\gu: 0$ \text{ and }
                $\gn': (0,1)\oplus (0,1)$  \end{tabular} \\ \hline\hline

8 & \begin{tabular}{l} $U(1)\cdot Sp(r,s)\cdot U(1),\, r+s=n$ \\
		$((\gh_{2n;\C})) \oplus ((\gh_{2n;\C}))$ \end{tabular}
	& \begin{tabular}{l} $\gv: (4r,4s)\oplus (4r,4s)$ \\
                $\gu: 0$ \text{ and }
                $\gn': (1,0)\oplus (1,0)$  \end{tabular} \\ \hhline{|~|-|-|}
  & \begin{tabular}{l} $\R^+ \cdot Sp(r,s)\cdot U(1),\, r+s=n$ \\
                $((\R + \R^{2n,2n})) 
			\oplus ((\gh_{2n;\C}))$ \end{tabular}
        & \begin{tabular}{l} $\gv: (4r,4s)\oplus (4r,4s)$ \\
                $\gu: 0$ \text{ and }
                $\gn': (0,1)\oplus (1,0)$  \end{tabular} \\ \hhline{|~|-|-|}
  & \begin{tabular}{l} $\R^+ \cdot Sp(r,s)\cdot \R^+,\, r+s=n$ \\
                $((\R + \R^{2n,2n})) 
		  \oplus ((\R + \R^{2n,2n}))$ \end{tabular}
        & \begin{tabular}{l} $\gv: (4r,4s)\oplus (4r,4s)$ \\
                $\gu: 0$ \text{ and }
                $\gn': (1,0)\oplus (1,0)$  \end{tabular} \\ \hhline{|~|-|-|}
  & \begin{tabular}{l} $U(1)\cdot Sp(n;\R)\cdot U(1)$ \\
		$((\gh_{2n;\C})) \oplus ((\gh_{2n;\C}))$ \end{tabular}
        & \begin{tabular}{l} $\gv: (2n,2n)\oplus (2n,2n)$ \\
		$\gu: 0$ \text{ and }
                $\gn': (1,0)\oplus (1,0)$  \end{tabular} \\ \hline\hline

9 & \begin{tabular}{l} $Sp(1)\cdot Sp(r,s)\cdot \{U(1),\, \R^+\}$ \\
		$((\gh_{n;\H})) \oplus ((\gh_{2n;\C})),\, r+s=n$ \end{tabular}
	& \begin{tabular}{l} $\gv: (4r,4s)\oplus (4r,4s)$ \\
                $\gu: 0$ \text{ and }
                $\gn': (3,0)\oplus (1,0)$  \end{tabular} \\ \hhline{|~|-|-|}
  & \begin{tabular}{l} $Sp(1;\R)\cdot Sp(n;\R)\cdot U(1)$ \\
                $((\gh_{n;\H})) \oplus ((\gh_{2n;\C}))$ \end{tabular}
        & \begin{tabular}{l} $\gv: (2n,2n)\oplus (2n,2n)$\\
                $\gu: 0$ \text{ and }
                $\gn': (1,2)\oplus (1,0)$  \end{tabular} \\ \hline\hline

10 & \begin{tabular}{l} $H_r = Sp(1)\cdot Sp(r,s)\cdot Sp(1)$ \\
		$((\gh_{n;\H})) + ((\gh_{n;\H})),\, r+s=n$ \end{tabular}
	& \begin{tabular}{l} $\gv: (4r,4s)\oplus (4r,4s)$ \\
                $\gu: 0$ \text{ and }
                $\gn': (3,0)\oplus (3,0)$  \end{tabular} \\ \hhline{|~|-|-|}
  & \begin{tabular}{l} $Sp(1;\R)\cdot Sp(n;\R)\cdot Sp(1;\R)$ \\
                $((\gh_{n;\H})) \oplus ((\gh_{n;\H}))$ \end{tabular}
        & \begin{tabular}{l} $\gv: (2n,2n)\oplus (2n,2n)$\\
                $\gu: 0$ \text{ and }
                $\gn': (1,2)\oplus (1,2)$  \end{tabular} \\ \hline\hline

11 & \begin{tabular}{l} $Sp(k,\ell)\cdot\{Sp(1),U(1),\{1\}\}\cdot Sp(r,s)$\\
     $((\gh_{n;\H})) \oplus \H^{(k,\ell)\times (r,s)},
	\begin{smallmatrix}k+\ell=m\\ r+s=m \end{smallmatrix}$ \end{tabular}
	& \begin{tabular}{l} $\gv: (4k,4\ell)$\\
		$\gu: (4kr+4\ell s, 4ks+4\ell r)$ \\
		$\gn': (3,0)$ \end{tabular} \\ \hhline{|~|-|-|}
   & \begin{tabular}{l} $Sp(m;\R)\cdot\{Sp(1;\R),U(1)\}\cdot Sp(n;\R)$\\
		$((\gh_{m;\H})) + \H^{m\times n}$ \end{tabular}
	&\begin{tabular}{l} $\gv: (2m,2m)$ \\
		$\gu: (2mn,2mn)$ \text{ and }
		$\gn': (2,1)$ \end{tabular} \\ \hline\hline

12 & \begin{tabular}{l} $Sp(k,\ell)\cdot\{Sp(1),U(1),\{1\}\},\, k+\ell=m$  \\
        $((\gh_{m;\H})) \oplus \Re \H^{(k,\ell)\times (k,\ell)}_0$ \end{tabular}
	& \begin{tabular}{l} $\gv: (4k,4\ell)$\\
		$\gu: (2m^2-m-1-4k\ell,4k\ell)$\\
		$\gn': (3,0)$ \end{tabular} \\ \hhline{|~|-|-|}
   & \begin{tabular}{l} $H=Sp(m;\R)\cdot\{Sp(1;\R),U(1)\}$ \\
	$((\gh_{m;\H})) \oplus \Re \H^{m\times m}_{sp,0}$ \end{tabular}
	& \begin{tabular}{l} $\gv: (2m,2m)$\\
		$\gu: (m^2-1,m^2-m)$ \\
		$\gn': (2,1)$ \end{tabular} \\ \hline\hline

13 & \begin{tabular}{l} $Spin(k,\ell)\cdot \{\{1\},\, SO(r,s)\}$ \\
        $((\gh_{1;\O})) \oplus \R^{(k,\ell)\times (r,s)}$ \\ 
		$k+\ell = 7, \ell \leqq k,  
			(r,s) = (2,0) \text{ or } (1,1)$ 
		\end{tabular}
	& \begin{tabular}{l} $\gv: (q,8-q),\, q=2[\frac{k+1}{2}]$\\
		$\gu: (rk+s\ell,r\ell+sk)$\\
		$\gn': (k,\ell)$ \end{tabular} \\ \hline\hline

14 & \begin{tabular}{l} $U(1)\cdot Spin(k,\ell), \, 
		k+\ell = 7, \ell \leqq k$\\
	$((\gh_{7;\C})) \oplus \R^{q,8-q},\, q=2[\frac{k+1}{2}]$
		\end{tabular}
   &  \begin{tabular}{l} $\gv: (2k, 2\ell)$ \\
		$\gu: (q,8-q)$ \text{ and }
		$\gn': (1,0)$ \end{tabular} \\ \hhline{|~|-|-|}
  & \begin{tabular}{l} $\R^+\cdot Spin(k,\ell), \, 
                k+\ell = 7, \ell \leqq k$\\
        $((\R + \R^{1,1}\otimes_\R \R^{k,\ell})) 
		\oplus \R^{q,8-q},\, q=2[\frac{k+1}{2}]$ \end{tabular}
   &  \begin{tabular}{l} $\gv: (7, 7)$ \\
                $\gu: (q,8-q)$ \text{ and }
                $\gn': (0,1)$ \end{tabular} \\ \hline\hline

15 & \begin{tabular}{l} $U(1)\cdot Spin(k,\ell),\, 
		k+\ell = 7, \ell \leqq k$\\
	$((\gh_{8;\C})) \oplus \R^{k,\ell}$ \end{tabular}
   & \begin{tabular}{l} $\gv: (2q,16-2q),\, q=2[\frac{k+1}{2}]$ \\
		$\gu: (k,\ell)$ \text{ and }
		$\gn': (1,0)$ \end{tabular} \\ \hhline{|~|-|-|}
   & \begin{tabular}{l} $\R^+\cdot Spin(k,\ell),\, 
                k+\ell = 7, \ell \leqq k$\\
	$((\R + \R^{1,1}\otimes_\R \R^{q,8-q})) 
        	\oplus \R^{k,\ell},\, q=2[\frac{k+1}{2}]$ \end{tabular}
   & \begin{tabular}{l} $\gv: (8,8)$ \\
                $\gu: (k,\ell)$ \text{ and }
                $\gn': (0,1)$ \end{tabular} \\ \hline\hline

16 & \begin{tabular}{l} $U(1)\cdot Spin(k,\ell)\cdot U(1),\, 
			k+\ell=8, \ell \leqq k$, \\
             $((\gh_{8;\C})) \oplus ((\gh_{8;\C}))$ \end{tabular}
	& \begin{tabular}{l} $\gv: (2k,2\ell)\oplus(2k,2\ell)$ \\
		$\gu: 0$\text{ and }
		$\gn': (1,0)\oplus(1,0)$ \end{tabular}\\  \hhline{|~|-|-|}
   & \begin{tabular}{l} $\R^+ \cdot Spin(k,\ell)\cdot U(1),\, 
                        k+\ell=8, \ell \leqq k$, \\
             $((\R + \R^{1,1}\otimes \R^{k,\ell})) \oplus 
		((\gh_{8;\C}))$ \end{tabular}
        & \begin{tabular}{l} $\gv: (8,8)\oplus(2k,2\ell)$ \\
                $\gu: 0$\text{ and }
                $\gn': (0,1)\oplus(1,0)$ \end{tabular}\\  \hhline{|~|-|-|}
  & \begin{tabular}{l} $\R^+ \cdot Spin(k,\ell)\cdot \R^+,\, 
                        k+\ell=8, \ell \leqq k$, \\
             $((\R + \R^{1,1}\otimes \R^{k,\ell})) \oplus 
                ((\R + \R^{1,1}\otimes \R^{k,\ell}))$ \end{tabular}
        & \begin{tabular}{l} $\gv: (8,8)\oplus(8,8)$ \\
                $\gu: 0$\text{ and }
                $\gn': (0,1)\oplus(0,1)$ \end{tabular}\\ \hhline{|~|-|-|}
  & \begin{tabular}{l} $U(1)\cdot SO^*(8)\cdot U(1)$ \\
	$((\gh_{8;\C})) \oplus ((\gh_{8;\C}))$ \end{tabular}
	& \begin{tabular}{l} $\gv: (8,8)\oplus(8,8)$ \\
		$\gu: 0$ \text{ and }
		$\gn': (1,0)\oplus(1,0)$ \end{tabular}\\ \hline\hline

17 & \begin{tabular}{l} $U(1)\cdot Spin(2k,2\ell),\, 
		k=3,4,5;\, \ell =5-k$ \\
		$((\gh_{16;\C})) \oplus \R^{2k,2\ell}$ \end{tabular}
	& \begin{tabular}{l} $\gv:  (q,16-q), q=2^{[\tfrac{k+1}{2}]+2}$\\
		$\gu: (2k,2\ell)$ \text{ and }
		$\gn': (1,0)$ \end{tabular}\\ \hhline{|~|-|-|}
	& \begin{tabular}{l} $\R^+ \cdot Spin(2k-1,2\ell+1),
		\begin{smallmatrix}k=3,4,5;\,  \ell = 5-k \end{smallmatrix}$\\
		$((\R + \R^{16,16})) \oplus 
			\R^{2k-1,2\ell+1}$ \end{tabular}
	& \begin{tabular}{l} $\gv: (16,16) $\\
		$\gu: (2k-1,2\ell+1)$\\
		$\gn': (0,1)$ \end{tabular}\\ \hhline{|~|-|-|}
   & \begin{tabular}{l} $U(1)\cdot Spin^*(10)$ \\
		 $((\gh_{16;\C})) \oplus \C^5$ \end{tabular}
	& \begin{tabular}{l} $\gv: (16,16)$\\
		$\gu: (10,0)$\text{ and }
		$\gn': (1,0)$ \end{tabular}\\ \hline\hline	

18 & \begin{tabular}{l} $\{SU(k,\ell),U(k,\ell),U(1)Sp(\tfrac{m}{2})\} 
		\cdot SU(r,s)$ \\
	 $((\gh_{2m;\C})) + \gs\gu(r,s),\, k+\ell=m, r+s=2$\end{tabular}
	& \begin{tabular}{l}  $\gv: (2kr+2\ell s,2ks+2\ell r)$\\
		$\gu: (3-2rs,2rs)$ \text{ and }
		$\gn': (1,0)$ \end{tabular} \\  \hhline{|~|-|-|}
   & \begin{tabular}{l} $\{SL(m;\R),GL(m;\R)\}\cdot SL(2;\R)$\\
	$((\R +\R^{1,1}\otimes\R^{m\times 2})) \oplus \gs\gl(2;\R)$\end{tabular}
	& \begin{tabular}{l}  $\gv: (2m,2m)$ \\
		$\gu: (1,2)$ \text{ and }
		$\gn': (0,1)$ \end{tabular} \\  \hhline{|~|-|-|}
   & \begin{tabular}{l} $\{SL(m/2;\H), GL(m/2;\H)\}\cdot SL(1;\H)$ \\
	$((\R + \H^{m/2,m/2})) \oplus \gs\gl(1;\H)$ \end{tabular}
	& \begin{tabular}{l}  $\gv: (2m,2m)$\\
		$\gu: (3,0)$ \text{ and }
		$\gn': (0,1)$\end{tabular}\\  \hhline{|~|-|-|}
  & \begin{tabular}{l} $Sp(k/2, \ell/2)\cdot GL(2;\R)$ \\
		$((\R + \R^{1,1}\otimes_\R \H^{k/2,\ell/2})) 
			\oplus \gs\gl(2;\R)$\end{tabular}
	& \begin{tabular}{l}  $\gv: (2m,2m)$\\
		$\gu: (1,3)$ \text{ and }
		$\gn': (0,1)$ \end{tabular}\\  \hhline{|~|-|-|}
  & \begin{tabular}{l} $Sp(m/2;\R)\cdot GL(1;\H)$ \\
		$((\gh_{2m;\C})) \oplus \gs\gl(2;\R)$\end{tabular}
        & \begin{tabular}{l} $\gv: (2m,2m)$ \\
		$\gu: (1,2)$ \text{ and }
		$\gn': (0,1)$ \end{tabular} \\  \hline\hline

19 & \begin{tabular}{l} $\{SU(k,\ell),U(k,\ell),U(1)Sp(k/2,\ell/2)\}\cdot $\\
	$\qquad \cdot U(r,s),\, k+\ell = m,  r+s=2 $ \\
	$((\gh_{2m;\C})) \oplus ((\gh_{2;\C}))$ \end{tabular}
   & \begin{tabular}{l} $\gv: (2kr+2\ell s,2ks+2\ell r)\oplus (2r,2s)$ \\
		$\gu: 0$ \text{ and }
		$\gn': (1,0)\oplus (1,0)$ \end{tabular} \\ \hhline{|~|-|-|}
    & \begin{tabular}{l} $\{SL(m;\R),GL(m;\R)\}\cdot GL(2;\R)$ \\
         $((\R + \R^{2m,2m})) \oplus ((\R + \R^{2,2}))$ \end{tabular}
   &\begin{tabular}{l} $\gv: (2m,2m)\oplus(2,2)$ \\
		$\gu: 0$ \text{ and }
		$\gn': (1,0)\oplus (1,0)$ \end{tabular} \\ \hhline{|~|-|-|}

   & \begin{tabular}{l} $\R^+\cdot Sp(k/2,\ell/2)\cdot GL(2;\R)$\\
        $((\Im\C + \C^{2k,2\ell} \otimes_\R\R^{1,1})) \oplus 
			((\Im\C + \R^{2,2}))$ \end{tabular}
     & \begin{tabular}{l} $\gv: (2m,2m)\oplus(2,2)$ \\
		$\gu: 0$\text{ and }
		$\gn': (1,0)\oplus (1,0)$ \end{tabular} \\ \hhline{|~|-|-|}
   & \begin{tabular}{l} $U(1)Sp(m/2;\R)\cdot U(r,s)$ \\
	$((\gh_{2m;\C})) \oplus ((\gh_{2;\C}))$ \end{tabular} 
    &\begin{tabular}{l} $\gv: (2m,2m)\oplus (2r,2s)$ \\
		$\gu: 0$ \text{ and }
		$\gn': (1,0)\oplus (1,0)$ \end{tabular} \\ \hline \hline

20 & \begin{tabular}{l} $\{SU(k,\ell),U(k,\ell),U(1)
			Sp(\tfrac{k}{2},\tfrac{\ell}{2})\}\cdot SU(a,b)\cdot$ \\
             $\qquad\cdot\{SU(r,s),U(r,s),U(1)Sp(\tfrac{r}{2},\tfrac{s}{2})\}$\\
	     $((\gh_{2m;\C})) \oplus ((\gh_{2n;\C})), \begin{smallmatrix}
	         k+\ell = m, a+b = 2, r+s = n \end{smallmatrix}$  \end{tabular}
	&\begin{tabular}{l} $\gv: (2(ak+b\ell),2(a\ell +bk))\oplus $\\
		$\qquad (2(ar+bs),2(as+br))$ \\
	$\gu: 0 \text{ and } \gn': (1,0)\oplus (1,0)$ \end{tabular} 
		\\ \hhline{|~|-|-|}
   & \begin{tabular}{l} $\{SU(k,\ell),U(k,\ell),U(1)
                        Sp(\tfrac{k}{2},\tfrac{\ell}{2})\}\cdot SU(a,b)\cdot$ \\
			$\qquad\cdot U(1)Sp(\tfrac{n}{2};\R)\}$ \\
		$((\gh_{2m;\C})) \oplus ((\gh_{2n;\C}))$ \end{tabular}
	& \begin{tabular}{l} $\gv: (2(ak+b\ell),2(a\ell +bk)) \oplus (2n,2n)$ \\
		$\gu: 0 \text{ and } \gn': (1,0)\oplus (0,1)$ \end{tabular}
                \\ \hhline{|~|-|-|}
   & \begin{tabular}{l} $U(1)Sp(\tfrac{m}{2};\R)\cdot SU(a,b)\cdot 
			U(1)Sp(\tfrac{n}{2};\R)$ \\
		$((\gh_{2m;\C})) \oplus ((\gh_{2n;\C}))$ \end{tabular}
        & \begin{tabular}{l} $\gv: (2m,2m) \oplus (2n,2n)$\\
		$\gu: 0 \text{ and } \gn': (0,1)\oplus (0,1)$ \end{tabular}
                \\ \hhline{|~|-|-|}
     & \begin{tabular}{l}$\{SL(m;\R),GL(m;\R)\}\cdot SL(2;\R)\cdot$ \\
           $\qquad\cdot\{SL(n;\R),GL(n;\R), 
			\R^+Sp(\tfrac{r}{2},\tfrac{s}{2})\}$ \\
		$ ((\gh_{2m;\C})) \oplus ((\gh_{2n;\C}))$\end{tabular}
	& \begin{tabular}{l} $\gv: (2m,2m) \oplus (2n,2n)$\\
		$\gu: 0 \text{ and } \gn': (0,1)\oplus (0,1)$ \end{tabular}
                \\ \hhline{|~|-|-|}
    & \begin{tabular}{l} $\R^+Sp(\tfrac{k}{2},\tfrac{\ell}{2})\cdot 
		SL(2;\R)\cdot \R^+ Sp(\tfrac{r}{2},\tfrac{s}{2})$ \\
                $((\gh_{2m;\C})) \oplus ((\gh_{2n;\C}))$ \end{tabular}
        & \begin{tabular}{l} $\gv: (2m,2m) \oplus (2n,2n)$\\
                $\gu: 0 \text{ and } \gn': (0,1)\oplus (0,1)$ \end{tabular}
                \\ \hhline{|~|-|-|}
    & \begin{tabular}{l} $\{SL(m/2;\H), GL(\tfrac{m}{2};\H)\}\cdot 
			SL(1;\H)\cdot$ \\
        $\qquad\cdot\{SL(\tfrac{n}{2};\H), GL(\tfrac{n}{2};\H), 
			\R^+Sp(\tfrac{n}{2};\R)\}$ \\
         $((\R + \H^{\tfrac{m}{2}}\otimes_\R\R^{1,1}))\oplus 
                ((\R + \R^{1,1}\otimes_\R\H^{\tfrac{n}{2}}))$ \end{tabular}
        & \begin{tabular}{l} $\gv: (2m,2m)\oplus (2n,2n)$\\
		$\gu: 0 \text{ and } \gn': (0,1)\oplus (0,1)$ \end{tabular}
                \\ \hhline{|~|-|-|}
    & \begin{tabular}{l} $\R^+Sp(\tfrac{m}{2};\R)\cdot SL(1;\H)\cdot 
		\R^+Sp(\tfrac{n}{2};\R)$\\
         $((\R + \R^{2m,2m}))\oplus ((\R + \R^{2n,2n}))$ \end{tabular}
        & \begin{tabular}{l} $\gv: (2m,2m)\oplus (2n,2n)$\\
                $\gu: 0 \text{ and } \gn': (0,1)\oplus (0,1)$ \end{tabular}
	\\ \hline \hline

21 & \begin{tabular}{l} $\{SU(k,\ell),U(k,\ell),U(1)Sp(\tfrac{k}{2},
		\tfrac{\ell}{2})\} \cdot SU(a,b)\cdot $ \\ 
	  $\quad\cdot U(r,s),\quad \begin{smallmatrix}
	   k+\ell = m, a+b = 2, (r,s) = (4,0) \text{ or } (2,2)
			\end{smallmatrix}$\\
       $((\gh_{2m;\C}))\oplus ((\gh_{8;\C}))\oplus\R^{2r-2+s,s}$ \end{tabular}
	& \begin{tabular}{l} $\gv: (4k,4\ell)\oplus$\\
		$\quad \oplus (2ar+2bs,2as+2br)$\\
		$\gu: (2r-2+s,s)$ \\ $\gn': (1,0\oplus (1,0)$
		\end{tabular} \\ \hhline{|~|-|-|}
    &  \begin{tabular}{l} $\{SL(m;\R),GL(m;\R)\}\cdot SL(2;\R)\cdot GL(4;\R)$\\
	$((\R + \R^{2m,2m}))\oplus ((\R + \R^{8,8}))
		\oplus \R^{3,3}$ \end{tabular}
	 & \begin{tabular}{l} $\gv: (2m,2m)\oplus (8,8)$\\
		$\gu: (3,3) \text{ and } \gn': (0,1)\oplus (0,1)$
		\end{tabular} \\ \hhline{|~|-|-|} 
   & \begin{tabular}{l} $\{SL(\tfrac{m}{2};\H), GL(\tfrac{m}{2};\H)\}
			\cdot SL(1;\H)\cdot GL(2;\H)$\\
         $((\R + \H^{\tfrac{m}{2},\tfrac{m}{2}}))  
                \oplus ((\R + \H^{2,2})) \oplus \R^{5,1}$ 
		\end{tabular}
	& \begin{tabular}{l} $\gv: (2m,2m)\oplus (8,8)$\\
		$\gu: (5,1) \text{ and } \gn': (0,1)\oplus (0,1)$
                \end{tabular} \\ \hhline{|~|-|-|}
    & \begin{tabular}{l} $Sp(k/2,\ell/2)\cdot GL(2;\R)\cdot GL(4;\R)$\\
	$((\gh_{2m;\C})) \oplus
		((\R + \R^{8,8})) \oplus \R^{3,3}$ \end{tabular}
	& \begin{tabular}{l} $\gv: (4k,4\ell)\oplus (8,8)$\\
		$\gu: (3,3) \text{ and } \gn': (0,1)\oplus (0,1)$ \end{tabular}
	\\ \hhline{|~|-|-|}
    & \begin{tabular}{l} $Sp(\tfrac{m}{2};\R)\cdot U(a,b)\cdot U(r,s)$\\
	$((\R + \R^{2m,2m}))\oplus ((\gh_{8;\C}))
		\oplus\R^{2r-2+s,s}$ \end{tabular}
	& \begin{tabular}{l} $\gv: (2m,2m) \oplus $\\
		$\quad \oplus (2ar+2bs,2as+2br)$\\
		$\gu: (2r-2+s,s)$ \\ $\gn': (0,1)\oplus (1,0)$
		\end{tabular} \\ \hhline{|~|-|-|}
    & \begin{tabular}{l} $Sp(\tfrac{m}{2};\R)\cdot GL(1;\H)\cdot GL(2;\H)$\\
		$((\R + \R^{2m,2m}))\oplus ((\R + \H^{2,2})) \oplus \R^{5,1}$
		\end{tabular}
	& \begin{tabular}{l} $\gv: (2m,2m)\oplus (8,8)$\\
		$\gu: (5,1) \text{ and } \gn': (0,1)\oplus (0,1)$\end{tabular}
        \\ \hline\hline

22 & \begin{tabular}{l} $U(a,b)\cdot U(r,s), 
		\begin{smallmatrix} (a,b)=(2,0) \text{ or } (1,1)\\
                   (r,s)=(4,0)\text{ or }(2,2)\end{smallmatrix}$ \\ 
                $((\gh_{8;\C})) + \R^{2r-2+s,s} +\gs\gu(a,b)$ \end{tabular}
	& \begin{tabular}{l} $\gv: (ar+bs,as+br)$\\
		$\gu: (2r-2+s,s)\oplus (2a-1,2b)$\\
		$\gn': (1,0)$ \end{tabular} \\ \hhline{|~|-|-|}
  & \begin{tabular}{l} $GL(2;\R)\cdot  GL(4;\R)$ \\
		$((\R + \R^{8,8})) \oplus \R^{3,3} \oplus \gs\gl(2;\R)$ 
		\end{tabular}
	& \begin{tabular}{l} $\gv: (8,8)$ \\
		$\gu: (3,3)\oplus (1,2) \text{ and }
		\gn': (0,1)$ \end{tabular} \\  \hhline{|~|-|-|}
  & \begin{tabular}{l} $GL(1;\H)\cdot GL(2;\H)$ \\
		$((\R + \H^{2,2})) \oplus \R^{5,1} \oplus \gs\gl(1;\H)$ 
		\end{tabular}
        & \begin{tabular}{l} $\gv: (8,8)$ \\
		$\gu: (5,1)\oplus (3,0) \text{ and } \gn': (0,1)$
                \end{tabular}
        \\  \hline\hline

23 & \begin{tabular}{l} $U(k,\ell)\cdot U(a,b)\cdot U(r,s)$ \\
		$\begin{smallmatrix} (k,\ell), (r,s) = (4,0) \text{ or } (2,2);
			\text{ and } (a,b) = (2,0) \text{ or } (1,1)
			\end{smallmatrix}$\\ 
                 $\R^{2k-2+\ell,\ell}\oplus ((\gh_{8;\C}))
                \oplus ((\gh_{8;\C})) \oplus \R^{2r-2+s,s}$ \end{tabular}
	& \begin{tabular}{l} $\gv: (ak+b\ell,a\ell +bk)\oplus $\\
			$\qquad \oplus (ar+bs,as+br)$\\
		$\gu: (2k-2+\ell,\ell)\oplus (2r-2+s,s)$\\
		$\gn': (1,0)\oplus (1,0)$
		\end{tabular} \\  \hhline{|~|-|-|}
  & \begin{tabular}{l} $GL(4;\R)\cdot GL(2;\R)\cdot  GL(4;\R)$ \\
		$\R^{3,3}\oplus ((\R + \R^{8,8}))
                \oplus ((\R + \R^{8,8})) \oplus \R^{3,3}$ \end{tabular}
        & \begin{tabular}{l} $\gv: (8,8)\oplus  (8,8)$ \\
		$\gu: (3,3)\oplus (3,3)$\\
			$ \gn': (0,1)\oplus (0,1)$
		\end{tabular}   \\   \hhline{|~|-|-|}
  & \begin{tabular}{l} $GL(2;\H)\cdot GL(1;\H)\cdot GL(2;\H)$ \\
		$\R^{5,1}\oplus ((\R + \H^{2,2}))
		\oplus ((\R + \H^{2,2})) \oplus\R^{5,1}$ \end{tabular}
        & \begin{tabular}{l} $\gv: (8,8)\oplus  (8,8)$ \\
		$\gu: (5,1)\oplus  (5,1)$\\
		$\gn': (0,1)\oplus (0,1)$ \end{tabular} \\ \hline\hline

24 & \begin{tabular}{l} $U(1)\cdot SU(k,\ell)\cdot U(1), \begin{smallmatrix} 
		(k,\ell) = (4,0) \text{ or } (2,2)\end{smallmatrix}$ \\
             $((\gh_{4;\C}))\oplus ((\gh_{4;\C})) \oplus \R^{2k-2+\ell,\ell}$ 
		\end{tabular}
	& \begin{tabular}{l} $\gv: (2k,2\ell)\oplus (2k,2\ell)$\\
		$\gu: (2k-2+\ell,\ell)$ \\
		$\gn': (1,0)\oplus(1,0)$ \end{tabular} \\  \hhline{|~|-|-|}
  & \begin{tabular}{l} $\R^+\cdot SL(4;\R)\cdot \R^+$\\
		$((\R + \R^{4,4})) \oplus ((\R + \R^{4,4})) \oplus \R^{3,3}$
			\end{tabular} 
        & \begin{tabular}{l} $\gv: (4,4)\oplus (4,4)$ \\
		$\gu: (3,3) \text{ and } \gn': (0,1)\oplus (0,1)$ \end{tabular}
		\\ \hline\hline

25 & \begin{tabular}{l} $\{\{1\},\,U(1)\}\cdot SU(k,\ell)\cdot \{\{1\},U(1)\}$\\
        $((\gh_{4;\C})) + \R^{k(k-1)+\ell (\ell-1),2k\ell},\,
		\begin{smallmatrix} k+\ell = 4 \end{smallmatrix}$ 
		\end{tabular}
   	& \begin{tabular}{l} $\gv: (2k,2\ell)$\\
		$\gu: (k(k-1)+\ell (\ell-1),2k\ell)$\\
		$\gn': (1,0)$ \end{tabular} \\   \hhline{|~|-|-|}
    & \begin{tabular}{l} $\{\{1\},\,\R^+\}\cdot SL(4;\R) \cdot \{\{1\},\,\R^+\}$ \\
		$((\R + \R^{4,4}))\oplus \R^{6,6}$ \end{tabular}
	& \begin{tabular}{l} $\gv: (4,4)$\\
		$\gu: (6,6) \text{ and } \gn': (0,1)$ \end{tabular} \\
\end{longtable}
\noindent
All the spaces $G_r/H_r = (N_r\rtimes H_r)/H_r$\,, corresponding to
entries of Table \ref{max-irred}, are weakly symmetric Riemannian 
manifolds except
entry 11 with $H_r = Sp(m)\times Sp(n)$, entry 12 with $H_r = Sp(m)$, 
entry 13 with $H_r = Spin(7) \times (\{1\} \text{ or }SO(2)$, and entry 25 
with $H_r = (\{1\} \text{ or } U(1))$.  In those four cases $G_r/H_r$
is not weakly symmetric.  See \cite[Theorem 15.4.12]{W2007}.

We now extract special signatures from Table \ref{indecomp}.  In order to
avoid redundancy we consider $SO(n)$ only for $n\geqq 3$, $SU(n)$ and $U(n)$
only for $n\geqq 2$, and $Sp(n)$ only for $n\geqq 1$.

\begin{corollary}\label{Lorentz2}
The Lorentz cases, signature of the form $(p-1,1)$ in
{\rm Table \ref{indecomp}}, all are weakly symmetric.  In addition
to their invariant Lorentz metrics,  all 
except $H = GL(1;\R)$ in Case 1 and 
$H = \R^+\cdot SL(2;\R)\cdot \R^+$ in Case 3 have
invariant weakly symmetric Riemannian metrics.  They are

Case 1. $H = U(n)$ with metric on $G/H$ of signature
$(n^2+2n-1,1)$, $H = GL(1;\R)$ with metric on $G/H$ of signature
$(2,1)$. 

Case 2. $H = U(4)$ with metric on $G/H$ of signature
$(26,1)$.

Case 3. $H = U(1)\cdot SU(n)\cdot U(1)$ with metric on 
$G/H$ of signature $(n^2+n+1,1)$, 
$H = \R^+\cdot SL(2;\R)\cdot \R^+$ with metric on
$G/H$ of signature $(3,1)$.

Case 4.  $H = SU(4)$ with metric on $G/H$ of signature
$(20,1)$.

Case 6.  $H = S(U(4)\times U(n))$ with metric on $G/H$ of 
signature  $(8n+6,1)$. 

Case 8.  $H = U(1)\cdot Sp(n)\cdot U(1)$ with metric on 
$G/H$ of signature $(8n+1,1)$

Case 9.  $H = Sp(1)\cdot Sp(n)\cdot U(1)$ with metric on
$G/H$ of signature $(8n+3,1)$

Case 14. $H = U(1)\cdot Spin(7)$ with metric on
$G/H$ of signature $(22,1)$

Case 15.  $H = U(1)\cdot Spin(7)$ with metric on
$G/H$ of signature $(23,1)$

Case 16.  $H = U(1)\cdot Spin(8)\cdot U(1)$ with metric on
$G/H$ of signature $(33,1)$

Case 17.  $H = U(1)\cdot Spin(10)$ with metric on $G/H$ of 
signature $(42,1)$

Case 18.  $H = \{SU(m),U(m),U(1)Sp(m/2)\} \cdot SU(2)$ with 
metric on $G/H$ of signature $(4m+3,1)$

Case 19.  $H = \{SU(m),U(m),U(1)Sp(m/2)\} \cdot U(2)$ with 
metric on $G/H$ of signature $(4m+5,1)$

Case 20.  $H = \{SU(m),U(m),U(1)Sp(m/2)\}\cdot SU(2) \cdot 
\{SU(n),U(n),U(1)Sp(n/2)\}$ with metric on $G/H$ of signature
$(4m+4n+1,1)$

Case 21.  $H = \{SU(n),U(n),U(1)Sp(\tfrac{n}{2})\}\cdot SU(2) \cdot U(4)$
with metric on $G/H$ of signature $(4n+23,1)$.

Case 22.  $H = U(2)\cdot U(4)$ with metric on $G/H$ of 
signature $(25,1)$

Case 23.  $H = U(4)\cdot U(2)\cdot U(4)$ with metric on
$G/H$ of signature $(45,1)$

Case 24.  $H = U(1)\cdot SU(4)\cdot U(1)$ with metric on
$G/H$ of signature $(23,1)$

Case 25.  $H = (U(1)\cdot)SU(4)(\cdot SO(2))$ with metric
on $G/H$ of signature $(20,1)$
\end{corollary}

\begin{corollary}\label{t-Lorentz2}
The complexifications of the Lorentz cases listed in {\rm Corollary
\ref{Lorentz2}} all are of trans--Lorentz signature $(p-2,2)$.
The trans--Lorentz cases, signature of the form $(p-2,2)$ in
{\rm Table \ref{indecomp}}, all are weakly symmetric and are as follows.

Case 1.  $H = GL(1;\R)$ with metric on $G/H$ of signature
$(1,2)$. 

Case 3. $H = U(1)\cdot SU(n)\cdot U(1)$ with metric on 
$G/H$ of signature $(n^2+n,2)$, $H = U(1)\cdot SU(1,1)\cdot U(1)$
with metric on $G/H$ of signature $(6,2)$, 
$H = \R^+\cdot SL(2;\R)\cdot \R^+$ with metric on
$G/H$ of signature $(2,2)$ 

Case 8.  $H = U(1)\cdot Sp(n)\cdot U(1)$ with metric on
$G/H$ of signature $(8n,2)$

Case 16.  $H = U(1)\cdot Spin(8)\cdot U(1)$ with metric on
$G/H$ of signature $(32,2)$

Case 19.  $H = \{SU(m),U(m),U(1)Sp(m/2)\} \cdot U(2)$ with 
metric on $G/H$ of signature $(4m+4,2)$

Case 20.  $H = \{SU(m),U(m),U(1)Sp(m/2)\}\cdot SU(2) \cdot 
\{SU(n),U(n),U(1)Sp(n/2)\}$ with metric on $G/H$ of signature
$(4m+4n,2)$

Case 21.  $H = \{SU(n),U(n),U(1)Sp(\tfrac{n}{2})\}\cdot SU(2) \cdot U(4)$
with metric on $G/H$ of signature $(4n+22,2)$.

Case 23.  $H = U(4)\cdot U(2)\cdot U(4)$ with metric on
$G/H$ of signature $(44,2)$

Case 24.  $H = U(1)\cdot SU(4)\cdot U(1)$ with metric on
$G/H$ of signature $(22,2)$
\end{corollary}


\end{document}